\documentclass[reqno,12pt]{article}

\usepackage{a4wide}
\usepackage{amsmath} 
\usepackage{amssymb}
\usepackage{amsthm}
\usepackage[utf8]{inputenc} 
\usepackage{graphicx} 
\usepackage[english, russian]{babel}

\numberwithin{equation}{section}

\newtheorem{questno*}{Question}
\newtheorem{theo}{Theorem}
\newtheorem*{theo*}{Theorem}

\newtheorem{prop}{Proposition}

\newtheorem{defi}{Definition}
\newtheorem*{defi*}{Definition}

\theoremstyle{remark}
\newtheorem{Remark}{Remark}

\newcommand{\N}{\mathbb{N}}

\newcommand{\Q}{\mathbb{Q}}
\newcommand{\R}{\mathbb{R}}
\renewcommand{\C}{\mathbb{C}}
\newcommand{\K}{\mathbb{K}}

\newcommand{\Qbar}{\overline{\mathbb Q}}

\newcommand{\etoile}{^\star}
\newcommand{\al}{\alpha}

\newcommand{\li}{\textup{Li}}

\newcommand{\eps}{\varepsilon}

\newcommand{\h}{{\bf H}}

\newcommand{\devasy}{\approx}
\newcommand{\hada}{\star}
\newcommand{\GG}{{\bf G}}
\newcommand{\calFbar}{\overline{\cal F}}
\newcommand{\horsspec}{{\cal S}}
\def\dd{\textup{d}}
\newcommand{\calD}{{\cal {D}}} 
\newcommand{\nvsing}{\Sigma}
\def\al{\alpha}
\newcommand{\deltarho}{\Delta_\rho}
\def\Gh{\widehat{\Gamma}}

\newcommand{\Netoile}{\mathbb N ^\ast}

\begin{document}

 \selectlanguage{english}

\title{On Siegel's problem for $E$-functions}
\date\today
\author{S. Fischler and T. Rivoal}
\maketitle

\begin{abstract} Siegel defined in 1929 two classes of power series, the $E$-functions and $G$-functions, which  generalize the Diophantine properties of the exponential and logarithmic functions respectively. In 1949, he asked whether any $E$-function can be represented as a polynomial with algebraic coefficients in a finite number of hypergeometric $E$-functions of the form ${}_pF_q(z^{q-p+1})$, $q\ge p\ge 1$, with rational parameters. The case of $E$-functions of differential order less than 2 was settled in the affirmative by Gorelov in 2004, but Siegel's question is open for higher order. We prove here that if Siegel's question has a positive answer, then the ring ${\bf G}$ of values taken by analytic continuations of $G$-functions at algebraic points must be a subring of the relatively ``small'' ring $\h$ generated by algebraic numbers, $1/\pi$ and the values of the derivatives of the Gamma function at rational points. Because that inclusion seems unlikely (and contradicts standard conjectures), this points towards a negative answer to Siegel's question in general. As intermediate steps, we first prove that any element of ${\bf G}$ is a coefficient of the asymptotic expansion of a suitable $E$-function, which completes previous results of ours. We then prove (in two steps) that the coefficients of the asymptotic expansion of an hypergeometric $E$-function with rational parameters are in $\h$. Finally, we prove a similar result for $G$-functions. 

\end{abstract}

\section{Introduction}

Siegel \cite[p. 223]{siegel} introduced in 1929 the notion of $E$-function as a generalization of the exponential and Bessel functions. We fix an embedding of $\Qbar$ into $\mathbb C$.
\begin{defi}\label{defi1}
A power series $F(z)=\sum_{n=0}^{\infty} \frac{a_n}{n!} z^n \in \Qbar[[z]]$ is an $E$-function if 
\begin{enumerate}
\item[$(i)$] $F(z)$ is solution of a non-zero linear differential equation with coefficients in 
$\Qbar(z)$.
\item[$(ii)$] There exists $C>0$ such that for any $\sigma\in \textup{Gal}(\Qbar/\mathbb Q)$ and any $n\ge 0$, $\vert \sigma(a_n)\vert \leq C^{n+1}$.
\item[$(iii)$] There exists $D>0$ and a sequence of integers $d_n$, with $1\le d_n \leq D^{n+1}$, such that
$d_na_m$ are algebraic integers for all~$m\le n$.
\end{enumerate}
\end{defi}
Siegel's original definition was in fact slightly more general than above and we shall make some remarks about this in \S\ref{subsec:comments}. 
Note that $(i)$ implies that the $a_n$'s all lie in a certain number field $\K$, so that in $(ii)$ there are only finitely many Galois conjugates $\sigma(a_n)$ of $a_n$ to consider, with $\sigma\in \textup{Gal}(\mathbb K/\mathbb Q)$ (assuming for simplicity that $\K$ is a Galois extension of $\Q$). 
$E$-functions are entire, and they form a ring stable under $\frac{d}{dz}$ and $\int_0^z$. A power series $\sum_{n=0}^{\infty} a_n z^n \in \Qbar[[z]]$ is said to be a $G$-function if $\sum_{n=0}^{\infty} \frac{a_n}{n!} z^n$ is an $E$-function. Algebraic functions over $\Qbar(z)$ regular at 0 and  polylogarithms (defined in \S\ref{subsecH}) are examples of $G$-functions.

The generalized hypergeometric series is defined as 
\begin{equation*}
{}_pF_{q}
\left[
\begin{matrix}
a_1, \ldots, a_p
\\
b_1, \ldots, b_q
\end{matrix}
\,; z \right] := \sum_{n=0}^\infty \frac{(a_1)_n\cdots (a_p)_n}{(1)_n(b_1)_n\cdots (b_q)_n} z^n
\end{equation*}
where $p,q\ge 0$ and $(a)_0:=1$, $(a)_n:=a(a+1)\cdots (a+n-1)$ if $n\ge 1$. The parameters $a_j$ and $b_j$ are in $\mathbb C$, with the restriction that $b_j\notin \mathbb Z_{\le 0}$ so that $(b_j)_n\neq 0$ for all $n\ge 0$. We shall also denote it by 
${}_pF_{q}[a_1, \ldots, a_p; b_1, \ldots, b_q; z]$. The special case $p=q$ is called the  {\em confluent} hypergeometric series.

Siegel proved that, for any integers $q\ge p\ge 1$, the series 
\begin{equation}
\label{eq:Efnhyp}
{}_{p}F_q \left[ 
\begin{matrix}
a_1, \ldots, a_{p}
\\
b_1, \ldots, b_p
\end{matrix}
; z^{q-p+1}
\right] 
\end{equation}
is an $E$-function (in the sense of this paper) when $a_j\in \mathbb Q$ and $b_j\in \mathbb Q\setminus \mathbb Z_{\le 0}$ for all $j$. He called them {\em hypergeometric $E$-functions}. The simplest examples are $\exp(z)=\sum_{n=0}^\infty \frac{z^n}{n!}={}_1F_1[1;1;z]$ and Bessel function $J_0(z):=\sum_{n=0}^{\infty} \frac{(iz/2)^{2n}}{n!^2}= {}_1F_2 [
1,1
; 
1
; (iz/2)^2]$. 
 If $a_j\in \mathbb Z_{\le 0}$ for some $j$, then the series reduces to a polynomial. Any polynomial with coefficients in $\Qbar$ of hypergeometric functions of the form ${}_{p}F_p[a_1, \ldots, a_p; b_1,\ldots, b_p ;\lambda z]$, with parameters
$a_j, b_j \in \mathbb Q$ and $\lambda\in \Qbar$, is an $E$-function.

The $E$-functions 
$$
L(z):=\sum_{n=0}^\infty \bigg(\sum_{k=0}^n \binom{n}{k}\binom{n+k}{n}\bigg)\frac{z^n}{n!}, \quad  H(z):=\sum_{n=0}^\infty \bigg(\sum_{k=1}^n \frac{1}{k}\bigg)\frac{z^n}{n!}
$$
are not of the hypergeometric type \eqref{eq:Efnhyp}, even with $z$ changed to $\lambda z$ for some $\lambda\in \Qbar$, but we have  
\begin{align*}
L(z)&= e^{(3-2\sqrt{2})z} \cdot {}_1F_1\big[1/2;1;4\sqrt{2}z \big],
\\ 
H(z)&=ze^{z} \cdot {}_2F_2\big[1,1;2,2;-z \big]. 
\end{align*} 
These puzzling identities, amongst others, naturally suggest to study further the role played by hypergeometric series in the theory of $E$-functions. In fact, Siegel had already stated in \cite{siegellivre} a problem that we reformulate as the following question.

\begin{questno*}[Siegel] 
Is it possible to write any $E$-function as a polynomial with coefficients in $\Qbar$ of hypergeometric $E$-functions of the form ${}_{p}F_q[a_1, \ldots, a_p; b_1,\ldots, b_p ;\lambda z^{q-p+1}]$, with parameters
$a_j, b_j \in \mathbb Q$ and $\lambda\in \Qbar$?
\end{questno*}

It must be understood that $\lambda$, $p,q$ and $q-p$ can take various values in the polynomial. Siegel's original statement is given in \S\ref{subsec:comments} along with some comments. 
Gorelov \cite[p. 514, Theorem~1]{gorelov} 
proved that the answer to Siegel's question is positive if the $E$-function (in Siegel's original sense) satisfies a linear differential equation with coefficients in $\Qbar(z)$ of order $\le 2$. He used the pioneering results of Andr\'e \cite{YA1} on $E$-operators. A version of Gorelov's theorem was reproved in \cite{rrsiegel} for $E$-functions as in Definition \ref{defi1} with a method also based on Andr\'e's results, but somewhat different in the details. It seems difficult to generalize any one of these two approaches when the order is $\ge 3$, though Gorelov~\cite{gorelov2} also obtained further  results in the case of $E$-functions solution of a linear inhomogeneous differential equation of order 2 with coefficients in $\Qbar(z)$, like $H(z)$ above.

\medskip

In this paper, we adopt another point of view on Siegel's question. Let us first define two subrings of $\mathbb C$; the former was introduced and studied in \cite{gvalues}.
\begin{defi} \hspace{0.1cm}
${\bf G}$ denotes the ring of $G$-values, i.e. the values taken at algebraic points by the analytic continuations of all $G$-functions. 

${\bf H}$ denotes the ring generated by $\Qbar$, $1/\pi$ and the values $\Gamma^{(n)}(r)$, $n\ge 0$, $r\in \mathbb Q\setminus \mathbb Z_{\le 0}$. 
\end{defi}
Here, $\Gamma(s):=\int_0^\infty t^{s-1}e^{-t}dt$ is the usual Gamma function that can be analytically continued to $\mathbb C \setminus\mathbb Z_{\le 0}$. We can now state our main result. 

\begin{theo} \label{thimplication}
At least one of the following statements is true:
\begin{enumerate}
\item[$(i)$] ${\bf G} \subset \h$; 
\item[$(ii)$] Siegel's question has a negative answer.
\end{enumerate}
\end{theo}

We provide in \S \ref{subsecH} another description of the ring $\h$, and explain there why the inclusion ${\bf G} \subset \h$ (and therefore a positive answer to Siegel's question) seems very unlikely; as  Y. Andr\'e, F. Brown and J. Fres\'an pointed out to us, this inclusion contradicts standard conjectures.

\medskip

The paper is organized as follows. In \S\ref{sec:comments}, we comment on Siegel's original formulation of his problem and make some remarks on the ring $\h$. In \S\ref{sec:Gcoeffs}, we prove that any element of ${\bf G}$ is a coefficient of the asymptotic expansion of a suitable $E$-function (Theorem \ref{thdevG}). In \S\ref{sec:2}, we prove that the coefficients of the asymptotic expansion of any hypergeometric series ${}_pF_p$ with rational parameters are in $\h$; then we generalize
this result to any  ${}_pF_q(z^{q-p+1})$ in \S \ref{sec:2bis}. We complete the proof of Theorem \ref{thimplication} in \S\ref{sec:fin} by comparing the results of the previous sections. Finally,  we consider in \S\ref{sec:siegelprobG} an analogous problem for $G$-functions and prove a similar result to Theorem \ref{thimplication}. We emphasize that the proof of Theorem~\ref{thimplication} uses, in particular, various results obtained in~\cite{gvalues} and~\cite{ateo}, the proofs of which are crucially based on a deep theorem due to Andr\'e, Chudnovsky and Katz on the structure of non-zero minimal differential equations satisfied by $G$-functions; see~\S\ref{ssec:Gconstconnec} and the references given there.

\medskip

\noindent {\bf Acknowledgements.} We warmly thank Yves Andr\'e, Francis Brown and Javier Fres\'an for their comments on a previous version of this paper, and in particular for explaining to us why the inclusion ${\bf G} \subset \h$ cannot hold under the standard conjectures on (exponential) periods. We also thank Javier Fres\'an and Peter Jossen for bringing to  our attention that Siegel's question was more general than we had originally thought.

\section{Comments on Theorem \ref{thimplication}}\label{sec:comments} 

\subsection{Siegel's formulation of his problem}\label{subsec:comments} 

In \cite[Chapter II, \S9]{siegellivre}, Siegel proved that the hypergeometric series of the type \eqref{eq:Efnhyp} with rational parameters are $E$-functions, and named them ``hypergeometric $E$-functions''. He then wrote on page 58: {\em Performing the substitution $x\mapsto \lambda x$ for arbitrary algebraic $\lambda$ and taking any polynomial in $x$ and finitely many hypergeometric $E$-functions, with algebraic coefficients, we get again an $E$-function satisfying a homogeneous linear differential equation whose coefficients are rational function of $x$. It would be interesting to find out whether all such $E$-functions can be constructed in the preceding manner.}

Siegel obviously considered $E$-functions in his sense, which we recall here: in Definition~\ref{defi1}, $(i)$ is unchanged but $(ii)$ and $(iii)$ have to be replaced by
\begin{enumerate}
\item[$(ii')$] For any $\varepsilon>0$ and for any $\sigma\in \textup{Gal}(\Qbar/\mathbb Q)$, there exists $N(\varepsilon, \sigma)\in \mathbb N$ such that for any $n\ge N(\varepsilon, \sigma)$, $\vert \sigma(a_n)\vert \leq n!^{\varepsilon}$.
\item[$(iii')$] There exists a sequence of integers $d_n\neq 0$ such that 
$d_na_m$ are algebraic integers for all~$m\le n$ and such that for any $\varepsilon>0$ there exists $N(\varepsilon)\in \mathbb N$ such that for any $n\ge N(\varepsilon)$, $\vert d_n\vert \le n!^{\varepsilon}$.
\end{enumerate}
Again, by $(i)$, there are only finitely many $\sigma$ to consider for a given $E$-function. We have chosen to formulate his problem for $E$-functions in the restricted sense of Definition~\ref{defi1} because the proof of Theorem~\ref{thimplication} is based on results which are currently proven only in this sense. However, {\em a fortiori}, Theorem~\ref{thimplication} obviously holds {\em verbatim} if one considers $E$-functions in Siegel's sense. Note also that the function $1-z$ is equal to the hypergeometric series ${}_1F_1[-1;1;z]$ so that Siegel could have formulated his problem in terms of hypergeometric series only, as we did.
Despite the apparences, the $E$-function $\sinh(z)=\frac{1}{2z}(e^{z}-e^{-z})$ is not a counter-example to Siegel's problem because $\frac{1}{2z}(e^z-1)={}_1F_1[1;2;z]$; there is no unicity of the representation of $E$-functions by polynomials in hypergeometric ones.

Furthermore, the series in \eqref{eq:Efnhyp} may be an $E$-function even if some of its parameters are not rational numbers. For instance, 
for every $\al\in \Qbar\setminus\mathbb{Z}_{\le 0}$, 
\begin{align*}
_{1}F_{1}\left[
\begin{matrix}
\alpha+1
\\
\al
\end{matrix} \,; z\right]&=\sum_{n=0}^\infty \frac{(\al+1)_n}{(1)_n(\al)_n} z^n = \sum_{n=0}^\infty \frac{\al+n}{\al} \cdot \frac{z^n}{n!} = \big(1+\frac{z}{\alpha}\big)e^z
\end{align*}
is an $E$-function. Thus, even though Siegel did not consider such examples, the notion of ``hypergeometric $E$-functions'' could be interpreted in a broader way than he did in his problem. Galochkin \cite{galochkin} proved the following non-trivial characterization, where $E$-functions are understood in Siegel's sense. (See \cite{rivoalgaloch} for a different proof restricted to hypergeometric $E$-functions of type ${}_{p}F_p$).
\begin{theo*}[Galochkin]\label{theo:galoch} Let $p, q\ge1$, $q\ge p$, $a_1, \ldots, a_{p}, b_1, \ldots, b_q \in (\mathbb C\setminus\mathbb Z_{\le 0})^{p+q}$ be such that $a_i\neq b_j$ for all $i,j$. Then, the hypergeometric series ${}_{p}F_q[a_1, \ldots, a_p; b_1,\ldots, b_q ;z^{q-p+1}]$ is an $E$-function if and only if the following two conditions hold:
\begin{enumerate}
\item[$(i)$] The $a_j$'s and $b_j$'s are all in $\Qbar$; 
\item[$(ii)$] The $a_j$'s and $b_j$'s which are not rational (if any) can be grouped in $k\le p$ pairs $(a_{j_1}, b_{j_1}), \ldots, (a_{j_k}, b_{j_k})$ such that $a_{j_\ell}-b_{j_\ell}\in \mathbb N$.
\end{enumerate} 
\end{theo*}
It follows that hypergeometric $E$-functions ${}_{p}F_q(z^{q-p+1})$ with arbitrary parameters are in fact $\Qbar$-linear combinations of hypergeometric $E$-functions ${}_{p'}F_{q'}(z^{q'-p'+1})$ (with various values of $p'$ and $q'$) with rational parameters.  Hence, there is no loss of generality in considering the latter instead of the former in Siegel's problem.

\subsection{The ring $\h$} \label{subsecH}
For $x\in \mathbb C\setminus\mathbb Z_{\le 0}$, we define the Digamma function 
$$
\Psi(x):=\frac{\Gamma'(x)}{\Gamma(x)}=-\gamma+\sum_{n=0}^\infty \Big(\frac{1}{n+1}-\frac{1}{n+x}\Big),
$$ 
where $\gamma$ is Euler's constant $\lim_{n\to +\infty} (\sum_{k=1}^n 1/k-\log(n))$, and the Hurwitz zeta function 
$$
\zeta(s,x):=\frac{(-1)^{s}}{(s-1)!}\Psi^{(s-1)}(x)=\sum_{n=0}^\infty \frac{1}{(n+x)^s}, \quad s\in  \mathbb N, \ s \geq 2.
$$ 
The polylogarithms are defined by
$$
\li_s(z) := \sum_{n=1}^\infty \frac{z^n}{n^s}, \quad s\in \Netoile = \mathbb N \setminus \{0\},
$$
where the series converges for $\vert z\vert \le 1$ (except at $z=1$ if $s=1$). The Beta function is defined as 
$$
\textup{B}(x,y):=\frac{\Gamma(x)\Gamma(y)}{\Gamma(x+y)}
$$
for $x,y\in \mathbb C$ which are not singularities of Beta coming from the poles of $\Gamma$ at non-positive integers.

In this section, we shall prove the following result.
\begin{prop} \label{prop:1} The ring $\h$ is generated by $\Qbar$, $\gamma$, $1/\pi$, $\li_s(e^{2i \pi r})$ $(s\in  \Netoile$, $r\in \mathbb Q$, $(s, e^{2i \pi r})\neq (1,1))$, $\log(q)$ $(q\in  \Netoile)$ and $\Gamma(r)$ $(r\in \mathbb Q\setminus\mathbb Z_{\le 0})$.

For any $r\in \mathbb Q\setminus\mathbb Z_{\le 0}$, $\Gamma(r)$ is a unit of $\h$.
\end{prop}

\begin{proof}
We first prove that for any $r\in \mathbb Q\setminus\mathbb Z_{\le 0}$, $\Gamma(r)$ is a unit of $\h$. Indeed, if $r\in  \Netoile$, then $\Gamma(r)\in  \Netoile$ and $1/\Gamma(r)\in \mathbb Q \subset \h.$ If $r\in \mathbb Q\setminus\mathbb Z$, then by the reflection formula \cite[p.~9, Theorem 1.2.1]{andrews}, we have 
$$
\frac{1}{\Gamma(r)} = \frac1{\pi}\sin(\pi r) \Gamma(1-r) \in \h
$$
because $1/\pi \in \h$, $\sin(\pi r)\in \Qbar \subset \h$ and $\Gamma(1-r)\in \h$. 

\medskip

From the identity $\Gamma'(x)= \Gamma(x)\Psi(x)$  we obtain that, for any integer $s\ge 1$ and any $r\in \mathbb Q\setminus\mathbb Z_{\le 0}$, 
$$
\Psi^{(s)}(r)=\frac{\Gamma^{(s+1)}(r)}{\Gamma(r)}-\sum_{k=0}^{s-1} {s \choose k} \frac{\Gamma^{(s-k)}(r)}{\Gamma(r)}\Psi^{(k)}(r).
$$
Since $\Gamma(r)$ is a unit of $\h$, we have $\psi(r)\in \h$ and it follows immediately by induction on $s $   that $\zeta(s,r)=\frac{(-1)^{s}}{(s-1)!}\Psi^{(s-1)}(r)\in \h$   for any $s\ge 2$ and any $r\in \mathbb Q\setminus\mathbb Z_{\le 0}$. In particular $\gamma=-\Psi(1)$ and the values of the Riemann zeta function $\zeta(s)=\zeta(s,1)$ $(s\ge 2)$ are all in~$\h$. Note that $\gamma$ is not expected to be in ${\bf G}$ but that $\zeta(s)\in {\bf G}$ for all $s\ge 2$.

We have for any $x\in \mathbb C\setminus\mathbb Z_{\le 0}$ and any $n\in \mathbb N$, 
\begin{equation}\label{gauss5}
\Psi(x+n)=\Psi(x)+\sum_{k=0}^{n-1} \frac1{k+x}, 
\end{equation}
and the identity $\Gamma'(x)= \Gamma(x)\Psi(x)$ also implies by induction that, for any $x\in \mathbb C\setminus\mathbb Z_{\le 0}$, we have 
\begin{equation}\label{gauss0}
\Gamma^{(s)}(x)= \Gamma(x) P_s\big(\Psi(x), \zeta(2,x), \ldots, \zeta(s,x)\big)
\end{equation}
for some $P_s\in \mathbb Q[X_1, \ldots, X_s]$. Furthermore, set $p, q \in \mathbb N$, $0<p\le q$, and $\mu:=\exp(2i\pi/q)$. Then, 
\begin{align}
\Psi\Big(\frac{p}{q}\Big)&=-\gamma-\log(q)-\sum_{n=1}^{q-1} \mu^{-np}\,\li_1(\mu^n),\label{gauss1}
\\
\li_1(\mu^p)&= -\frac{1}{q} \sum_{n=1}^q \mu^{np}\, \Psi\Big(\frac{n}{q}\Big),\quad p\neq q\label{gauss2}
\\
\zeta\Big(s,\frac{p}{q}\Big) &= q^{s-1}\sum_{n=1}^{q} \mu^{-np}\, \li_s(\mu^n), \quad s\ge 2\label{gauss3}
\\
\li_s(\mu^p) &= \frac{1}{q^s} \sum_{n=1}^q \mu^{np} \,\zeta\Big(s,\frac{n}{q}\Big),\quad s\ge 2.\label{gauss4}
\end{align}
We refer to  \cite[p.~14]{andrews} for details. 
From \eqref{gauss2} and \eqref{gauss4}, we deduce that $\li_s(\mu^p) \in \h$ for any $s\geq 1$ (with $(s,\mu^p)\neq(1,1)$); then  \eqref{gauss1} implies in turn that $\log(q)\in \h$. The numbers $\log(q)$ and $\li_s(\mu^p)$ are also in ${\bf G}$.

The set of Identities \eqref{gauss5}--\eqref{gauss4} shows that $\h$ coincides with the ring generated by $\Qbar$, $\gamma=-\Psi(1)$, $1/\pi$, $\li_s(e^{2i \pi r})$ ($s\in  \Netoile$, $r\in \mathbb Q$, $(s, e^{2i \pi r})\neq (1,1)$), $\log(q)$ $(q\in  \Netoile)$ and $\Gamma(r)$ $(r\in \mathbb Q\setminus\mathbb Z_{\le 0})$.
\end{proof}

Other units of $\h$ can be easily identified, which are  also units of ${\bf G}$ (see \cite[\S2.2]{gvalues}): non-zero algebraic numbers and the values of the Beta function $\textup{B}(x,y)$ at rational numbers $x,y$ at which it is defined and non-zero. It follows that $\pi=\Gamma(1/2)^2=\textup{B}(1/2, 1/2)$ and more generally 
$\Gamma(a/b)^b=(a-1)!\prod_{j=1}^{b-1} \textup{B}(a/b,ja/b)$, $a,b\in \Netoile$, are units of $\h$. By the Chowla-Selberg formula \cite[p.~230, Corollary 2]{moreno}, periods of CM elliptic curves defined over $\mathbb Q$ are also units of $\h$.

If Siegel's problem has a positive answer, Theorem  \ref{thimplication} yields ${\bf G} \subset \h$: any element of $ {\bf G}$ can be written as a polynomial, with algebraic coefficients, in the numbers $\gamma $, $1/\pi$, $\li_s(e^{2i \pi r})$, $\log(q)$  and $\Gamma(r)$ of Proposition \ref{prop:1}. This seems extremely doubtful: we recall that $ {\bf G}$ contains   all the multiple zeta values 
$$
\zeta(s_1, s_2, \ldots, s_n):=\sum_{k_1>k_2>\cdots>k_n\ge 1}\frac{1}{k_1^{s_1}k_2^{s_2}\cdots k_n^{s_n}},
$$
where the integers $s_j$ are such that $s_1\ge2, s_2\ge 1, \ldots, s_n\ge 1$, all values at algebraic points of (multiple) polylogarithms, all elliptic and abelian integrals, etc. For now, we have proved that ${\bf G} \cap \h$ contains the ring generated by $\Qbar$, $1/\pi$ and all the values $\li_s(e^{2i \pi r})$, $\log(q)$  and $\textup{B}(x,y)$, and it is in fact possible that both rings are equal. 

\medskip

It is interesting to know what can be deduced from the standard conjectures in the domain,  such as the Bombieri-Dwork conjecture ``$G$-functions come from geometry'', Grothen\-dieck's periods conjecture, its extension to exponential periods by Fres\'an-Jossen, and the Rohrlich-Lang conjecture on the values of the Gamma function; see  \cite[Partie~III]{andremotifs} and \cite[p.~201, Conjecture~8.2.5]{fresan}. In a private communication to the authors, Y. Andr\'e wrote the following argument, which he has autorised us to reproduce here. It shows that ${\bf G} \subset \h$ cannot hold under these standard conjectures:

\smallskip

{\em Because of the presence of $\gamma$, the inclusion ${\bf G} \subset \h$ does not contradict Grothendieck's periods conjecture but it certainly contradicts its extension to exponential motives.  More precisely, in the description of $\h$ given in Proposition \ref{prop:1}, we find $\gamma$ (a period of an exponential motive $E_\gamma$, which is a  non-classical extension of the Tate motive \cite[\S12.8]{fresan}), $1/\pi$ (a period of the Tate motive), $\li_s(e^{2i\pi r})$ (periods of a mixed Tate motive over $\mathbb Z[1/r]$), $\log(q)$ (a period of a $1$-motive over $\mathbb Q$), and $\Gamma(r)$ whose suitable powers are periods of Abelian varieties with complex multiplication by $\mathbb Q(e^{2i\pi r})$. On the one hand, let $M$ be the Tannakian category of mixed motives over  $\Qbar$ generated by all these motives.  
On the other hand, consider a non CM elliptic curve over $\Qbar$ and $E$ its motive. The periods of $E$ are in ${\bf G}$: indeed, it is enough to consider the Gauss hypergeometric solutions centered at $1/2$, and to observe that the periods of the fiber at $1/2$ of the Legendre family can be expressed using values of the Beta function at rational points by the Chowla-Selberg formula, and in particular are algebraic in $\pi$ and $\Gamma(1/4)$. If ${\bf G} \subset \h$, the periods of $E$ are in $\h$. By the exponential periods conjecture, $E$ would be in $M$, which is impossible since  the motivic Galois group of $M$ is pro-resoluble, while that of $E$ is $GL_2$.}

\medskip

We conclude this section with a question of J. Fres\'an: at which differential order can we expect to find a counter-example to Siegel's problem? Based on the above remarks, it seems unlikely  that all the values $\li_s(\alpha)$ are in $\h$, where the integer $s\ge 1$ and $\alpha \in\Qbar$, $\vert \alpha\vert <1$. From the proof of Theorem \ref{thdevG} below,  we deduce that if $\li_s(\alpha)\notin \h$, then the $E$-function 
$$
\sum_{n=2}^\infty \bigg(\sum_{k=1}^{n-1}\frac{\alpha^k}{k^s}\bigg) \frac{z^n}{n!}
$$
is such a counter-example. It is of differential order at most $s+2$ because it is in the kernel of the differential operator 
$
P(\theta-2)+zQ(\theta-1)+z^2R(\theta)
\in \Qbar[z,\frac{d}{dz}]$, where $\theta:=z\frac{d}{dz}$ and 
$$
P(x):=(x+2)(x+1)^{s+1}, \quad Q(x):=(x+1)(\alpha x^s-(x+1)^s), \quad R(x):=\alpha x^s.
$$
It is thus possible that a counter-example to Siegel's problem already exists at the order~3. However, the function $H(z):=\sum_{n=0}^\infty (\sum_{k=1}^{n}\frac{1}{k}) \frac{z^n}{n!}$ is an example of order 3  to the problem (see the Introduction) and this shows that one must be careful and not draw hasty conclusions here.

\section{Elements of $\GG$ as coefficients of asymptotic expansions of $E$-functions} \label{sec:Gcoeffs}

\subsection{Definition of asymptotic expansions} \label{subsecasyexp}

As in \cite{ateo}, the asymptotic expansions used throughout this paper are defined as follows.

\begin{defi} \label{defiasy}
Let $\theta\in\R$, and $\Sigma\subset\C$, $S\subset\mathbb C$, $T\subset\N$ be finite subsets. Given
complex numbers $ c_{\rho, \al,i,n}$, we write
\begin{equation} \label{eqasy1}
f(x) \devasy \sum_{\rho\in\Sigma} e^{\rho x} 
 \sum_{\al\in S } \sum_{i\in T } \sum_{n=0}^{\infty} c_{\rho, \al,i,n}x^{-n-\al}\log(1/x)^i
 \end{equation}
and say that the right-hand side is the asymptotic expansion of $f(x)$ in a large sector bisected by 
the direction $\theta$, if there exist $\eps, R, B, C > 0$ and, for any $\rho\in\Sigma$, a function $f_\rho(x)$ holomorphic in
$$U = \Big\{x\in\C, \,\,|x|\geq R, \, \, \theta-\frac{\pi}2-\eps \leq \arg(x) \leq \theta+\frac{\pi}2+\eps \Big\},
$$
such that 
$$f(x) = \sum_{\rho\in\Sigma} e^{\rho x} f_\rho(x)
 $$
 and
$$
\Big| f_\rho(x) - \sum_{\al\in S } \sum_{i\in T } \sum_{n=0}^{N-1} c_{\rho, \al,i,n}x^{-n-\al}\log(1/x)^i\Big| \leq C^N N! |x|^{B-N} 
$$
for any $x\in U$ and any $N\geq 1$.
\end{defi}

This means (see \cite[\S\S 2.1 and 2.3]{Ramis}) that for any $\rho\in\Sigma$, 
\begin{equation} \label{eqasy2}
 \sum_{\al\in S } \sum_{i\in T } \sum_{n=0}^{N-1} c_{\rho, \al,i,n}x^{-n-\al}\log(1/x)^i
 \end{equation}
 is 1-summable in the direction $\theta$ and its sum is $f_\rho(x)$. Using a result of 
 Watson (see \cite[\S 2.3]{Ramis}), the sum $f_\rho(x)$ is then determined by its asymptotic expansion \eqref{eqasy2}. 
Therefore the expansion on the right-hand side of \eqref{eqasy1} determines $f(x)$, 
up to analytic continuation. The converse holds too: \cite[Lemma 1]{ateo} asserts that a 
 given function $f(x)$ can have at most one asymptotic expansion in the sense of Definition \ref{defiasy}. Of course we assume implicitly (throughout this paper) that $\Sigma$, $S$ and $T$ in \eqref{eqasy1} cannot trivially be made smaller, and that for any $\alpha$ there exist $\rho$ and $i$ with $c_{\rho, \al,i,0}\neq 0$.

\subsection{Computing asymptotic expansions of $E$-functions}\label{subsecJEP}

In this section, we state \cite[Theorem 5]{ateo} which enables one to determine the asymptotic expansion of an $E$-function. We refer to \cite{ateo} for more details.

Let $E(x)=\sum_{n=1}^\infty a_n x^n$ be an $E$-function such that $E(0)=0$; consider $g(z)=\sum_{n=1}^\infty \frac{a_n}{z^{n+1}}$. Denoting by $\calFbar : \C[z,\frac{\dd}{\dd z}] \to \C[x,\frac{\dd}{\dd x}]$ the Fourier transform 
of differential operators, i.e. the morphism of $\C$-algebras defined by 
$\calFbar(z) = \frac{\dd}{\dd x}$ and $\calFbar(\frac{\dd}{\dd z}) = -x$, there exists a $G$-operator $\calD$ such that $\calFbar \calD E = 0$, and we have $(\frac{\dd}{\dd z })^\delta \calD g = 0$ where 
$\delta$ is the degree of $\calD$. We denote by 
$\nvsing$ the set of all finite singularities of $ \calD$ and let 
\begin{equation}\label{eq:horsspec}
\horsspec = \R\setminus\{\arg(\rho-\rho'), \rho,\rho'\in\nvsing, \rho\neq\rho'\}
\end{equation}
where all the values modulo $2\pi$ of the argument of $\rho-\rho'$ are considered, so that $\horsspec+\pi = \horsspec$.

We fix $\theta\in\R$ with $- \theta\in \horsspec$ (so that the direction $\theta$ is not anti-Stokes, i.e. not singular, see for instance \cite[p. 79]{Loday}). For any $\rho\in\nvsing$ we denote by $\deltarho = \rho - e^{-i\theta}\R_+$ the closed half-line of 
angle $-\theta+\pi \bmod 2\pi$ starting at $\rho$. Since
$-\theta\in\horsspec$, no singularity $\rho'\neq\rho$ of $ \calD$ lies on $\deltarho $: these 
half-lines are pairwise disjoint. We shall work in the simply connected cut plane obtained from $\C$ 
by removing the union of these half-lines. We agree that for $\rho \in \nvsing$ and $z$ in the 
cut plane, $\arg(z-\rho)$ will be chosen in the open interval $(-\theta-\pi,-\theta+\pi)$. This enables 
one to define $\log(z-\rho)$ and $(z-\rho)^\al$ for any $\al \in \Q$.

Now let us fix $\rho\in\nvsing$. Combining theorems of Andr\'e, Chudnovski and Katz (see \cite[\S3]{YA1} or \cite[\S4.1]{gvalues}), 
there exist (non necessarily distinct) rational numbers $t_1^\rho, \ldots, t_{J(\rho)}^\rho$, with $J(\rho)\geq 1$, 
and $G$-functions $g_{j,k}^\rho$, for $1\leq j \leq J(\rho) $ and $0\leq k \leq K(\rho,j)$, such that a basis of local solutions of $(\frac{\dd}{\dd z })^\delta \calD$ around $\rho$ (in the above-mentioned cut plane) 
is given by the functions
 \begin{equation}\label{eqdeffjk}
f_{j,k}^\rho(z-\rho) = (z-\rho)^{t_j^\rho} \sum_{k'=0}^k g_{j,k-k'}^\rho(z-\rho) \frac{\log(z-\rho)^{k'}}{k'!}
\end{equation}
for $1\leq j \leq J(\rho) $ and $0\leq k \leq K(\rho,j)$. Since $(\frac{\dd}{\dd z })^\delta \calD g= 0$ 
we can expand $g$ in this basis:
\begin{equation}\label{gconnectbaseack}
g(z) = \sum_{j=1}^{J(\rho)}\sum_{k=0}^{K(\rho,j)}\varpi_{j,k}^\rho f_{j,k}^\rho(z-\rho)
\end{equation}
with connection constants $\varpi_{j,k}^\rho$; Theorem 2 of \cite{gvalues} yields $\varpi_{j,k}^\rho\in\GG$.

We denote by $\{u\} \in[0,1)$ the fractional part of a real number $u$, and agree that all 
derivatives of this or related functions taken at integers will be right-derivatives. We let
\begin{equation} \label{eqgg2}
y_{\al,i}(z) = \sum_{n=0}^\infty \frac{1}{i!} \frac{\dd^{i}}{\dd t^{i}}\Big(\frac{\Gamma(1-\{t\})}
{\Gamma(-t-n)}\Big)_{| t=\al } z^n \in\Q[[z]]
\end{equation}
for $\al\in\Q$ and $i\in\N$. We also denote
 by $\hada$ the Hadamard (coefficientwise) product of formal power series in $z$, and we consider 
\begin{equation} \label{eqgg1}
\eta_{j,k}^\rho(1/x) = \sum_{m=0}^k (y_{t_j^\rho,m}\hada g_{j,k-m}^\rho)(1/x) \in \Qbar[[1/x]]
\end{equation}
for any $1\leq j \leq J(\rho)$ and $0\leq k \leq K(j,\rho)$. Then \cite[Theorem 5]{ateo} is the following result, where $\Gh:=1/\Gamma$.

\begin{theo}\label{thjep} 
In a large sector bisected by the direction $\theta$ 
we have the following asymptotic expansion:
\begin{multline} \label{eqthjep}
E(x) \devasy \\ \sum_{\rho\in\Sigma} e^{\rho x} \sum_{j=1}^{J(\rho)} \sum_{k=0}^{K(j,\rho)} \varpi_{j,k}^\rho x^{-t_j^\rho -1}
 \sum_{i = 0}^{ k} \Big( \sum_{\ell = 0} ^{k -i} \frac{(-1)^{\ell}}{\ell!} 
\Gh^{(\ell)}(1-\{t_j^\rho \}) \eta_{j, k-\ell-i}^\rho (1/x) \Big) \frac{ \log(1/x)^{i} }{i!}.
\end{multline}
\end{theo}
This theorem applies to any hypergeometric $E$-function ${}_{p}F_q[a_1, \ldots, a_p; b_1, \ldots, b_q; z^{q-p+1}]$ with rational parameters and $q\ge p\ge 1$.

\begin{Remark} \label{rem1}
 In the formula \eqref{eqthjep}, the term  corresponding to a triple  $(\rho,j,k)$   is zero whenever $ f_{j,k}^\rho(z-\rho)$ is holomorphic at $z=\rho$. 
This observation will be very useful in the proof of Theorem \ref{thhyp3} in \S \ref{ssec:2bis2}.
\end{Remark}

\subsection{$G$-values as coefficients of  asymptotic  expansions} \label{ssec:3.3}

We can now state and prove the main result of this section.

\begin{theo} \label{thdevG}
For any $\xi\in \GG$, there exists an $E$-function $E(z)$ such that for any $\theta\in[-\pi, \pi)$ outside a finite set, $\xi$ is a coefficient of the asymptotic expansion of $E(x)$ in a large sector bisected by $\theta$.
\end{theo}

\noindent {\em Proof.} Let $\xi\in \GG$; we may assume $\xi\neq0$. Using \cite[Theorem 1]{gvalues}, there exists a $G$-function $h(z)$ holomorphic at $z=1 $ such that $h(1)=\xi$. Let $g(z) = \frac{h(1/z)}{z(z-1)}$. This function has a Taylor expansion at $\infty$ of the form $\sum_{n=1}^{\infty} \frac{a_n}{z^{n+1}} $, and $E(x) = \sum_{n=1}^{\infty} \frac{a_n}{n!} x^{n} $ is an $E$-function. Using the results of \cite{ateo} recalled in \S \ref{subsecJEP} we shall compute (partially) its asymptotic expansion at infinity in a large sector bisected by the direction $\theta$, for any $\theta\in[-\pi, \pi)$ outside a finite set; we shall prove that the coefficient of $e^x$ in this expansion is equal to $\xi$. With this aim in mind, we keep the notation of \S \ref{subsecJEP}, including $\calD$ and $\theta$.
 
We let $\rho=1$ (eventhough we still write $\rho$ for better readability), and consider a basis of local solutions of $(\frac{\dd}{\dd z })^\delta \calD$ around $\rho$ with functions $f_{j,k}^\rho$ and $ g_{j,k}^\rho$ as in \S \ref{subsecJEP}. By Frobenius' method, upon shifting $t_j^\rho$ by an integer we may assume that $ g_{j,0}^\rho(0)\neq0$. Moreover, upon performing $\Qbar$-linear combinations of the basis elements and a permutation of the indices, we may assume that $t_1^\rho < \ldots < t_{J(\rho)}^\rho$ so that the solutions $f_{j,k}^\rho$ have pairwise distinct asymptotic behaviours at 0, namely 
$f_{j,k}^\rho(s) \sim \frac{g_{j,0}^\rho(0)}{k!} s^{t_j^\rho}\log(s)^k$. At last, dividing each $f_{j,k}^\rho$ with $ g_{j,0}^\rho(0)$ we may assume that $ g_{j,0}^\rho(0)=1$ for any $j$.

Now consider the expansion 
 \begin{equation}\label{eqdevbis}
g(z) = \sum_{j=1}^{J(\rho)}\sum_{k=0}^{K(\rho,j)}\varpi_{j,k}^\rho f_{j,k}^\rho(z-\rho).
\end{equation}
Let $T= \{(j,k), \, \varpi_{j,k}^\rho\neq0\}$. Since $g$ is not identically zero, $T$ is not empty. Let $j_0\in\{1,\ldots,J(\rho)\}$ be the minimal value such that $(j_0,k)\in T $ for some $k$, and let $k_0$ be the maximal value such that $(j_0,k_0)\in T $. Then on the right-hand side 
of Eq.~\eqref{eqdevbis}, the leading term as $z\to\rho$ is given by $(j,k)=(j_0,k_0)$, so that
 \begin{equation}\label{eqasy34}
g(z) \sim \frac{\varpi_{j_0,k_0}^\rho}{k_0!} (z-\rho)^{t_{j_0}^\rho} \log(z-\rho)^{k_0} 
\end{equation}
since $ g_{j_0,0}^\rho(0)=1$. Now recall that $g(z) = \frac{h(1/z)}{z(z-1)}$ with $h(1)=\xi\neq 0$ and $\rho=1$; therefore $g(z)\sim\frac{\xi}{z-1}$. Comparing this with 
Eq.~\eqref{eqasy34} yields $t_{j_0}^\rho=-1$, $k_0=0$, and $ \varpi_{j_0, 0}^\rho=\xi$.

Let us  consider the asymptotic expansion given by Theorem \ref{thjep}, and especially the coefficient of $e^x$ that we denote by $\alpha$. This coefficient comes from the multiple sum in Eq.~\eqref{eqthjep}. In this sum, we have $\varpi_{j,k}^\rho=0$ for any $j<j_0$ and any $k$ (by definition of $j_0$), so that these terms do not contribute to the value of $\alpha$. For any $j>j_0$ we have $t_{j }^\rho> t_{j_0}^\rho=-1$ so that $-t_{j }^\rho-1<0$ and the corresponding terms do not contribute either. Therefore the value of $\alpha$ is given only by the terms corresponding to $j=j_0$ (with $ t_{j_0}^\rho=-1$):
$$\alpha = \sum_{k=0}^{K(\rho,j_0)} \varpi_{j_0,k}^\rho \sum_{\ell=0}^k \frac{(-1)^\ell}{\ell!} \Gh^{(\ell)}(1)\eta_{j_0, k-\ell}^\rho (0).$$
Now recall that by definition, $k_0=0$ is the maximal value of $k$ such that $\varpi_{j_0,k}^\rho \neq0$. Therefore the previous sum has (at most) one non-zero term: the one corresponding to $k=0$. Since $\Gh(1)=1$ and $\varpi_{j_0,0}^\rho =\xi$ we have $\alpha = \xi \eta_{j_0, 0}^\rho (0) = \xi y_{-1,0}(0)g_{j_0,0}^\rho(0)=\xi$ using 
Eqs.~\eqref{eqgg2} and~\eqref{eqgg1}. This concludes the proof that the coefficient of $e^x$ in the asymptotic expansion of $E(x)$ is equal to $\xi$.

\section{Asymptotic expansion of the confluent hypergeometric series ${}_pF_p(z)$} \label{sec:2}

In this section, we prove the following result (recall that asymptotic expansions have been defined in \S \ref{subsecasyexp}). It will be generalized in \S\ref{sec:2bis} but it is important to prove it first.

\begin{prop} \label{thhyp} Let $\theta\in(-\pi,\pi)\setminus\{0\}$, and $f(z):= {}_{p}F_p[ 
a_1, \ldots, a_{p}; 
b_1, \ldots, b_p
; z] $ be a hypergeometric function with parameters $a_j\in \mathbb Q$ and $b_j\in \mathbb Q\setminus\mathbb Z_{\le 0}$. Then $f(z)$ has an asymptotic expansion 
$$
f(x) \devasy \sum_{\rho\in\Sigma} e^{\rho x} 
 \sum_{\al\in S } \sum_{i\in T } \sum_{n=0}^{\infty} c_{\rho, \al,i,n}x^{-n-\al}\log(1/x)^i
$$
 in a large sector bisected by $\theta$, with $\Sigma\subset \{0,1\}$, $S\subset \mathbb Q$ and $T\subset \mathbb N$ both finite, and coefficients $c_{\rho, \al,i,n}$ in $\h$.
 \end{prop}

\noindent {\em Proof.} If one of the $a_j$'s is in $\mathbb Z_{\le 0}$, the hypergeometric series is in $\mathbb C[z]$ and the conclusion clearly holds with $c_{\rho, \al,i,n}$ in $\Qbar$. From now on, as for the $b_j$'s, we assume that none of the $a_j$'s is in $\mathbb Z_{\le 0}$. 

Let 
$$
R(s)=R(\underline{a}, \underline{b};s):= \frac{\prod_{j=1}^p \Gamma(a_j+s)}{\prod_{j=1}^p \Gamma(b_j+s)} \Gamma(-s).
$$
The poles of $R(s)$ are located at $-a_j-k$, $k\in \mathbb Z_{\ge 0}$, $j=1, \ldots, p$, and at $\mathbb Z_{\ge 0}$. We define the series
$$
L_{p}(\underline{a}, \underline{b};z) := 
\sum_{j=1}^p\sum_{k=0}^\infty 
\textup{Residue}\big(R(s)z^{s},s=-a_j-k\big).
$$
Set  
$\nu:=\sum_{j=1}^p a_j -\sum_{j=1}^p b_j$, $b_{p+1}:=1$ and 
$$
e_{k,m}:=e_{k,m}(\underline{a}, \underline{b}):= \sum_{j=1}^{p+1}(1-\nu+ b_j+m)_{k-m}\frac{\prod_{i=1}^p (a_i-b_j)}{\prod_{i=1, i\neq j}^{p+1}(b_i-b_j)}.
$$
We define a sequence 
$
C_k:=C_k(\underline{a}, \underline{b}) 
$
by induction:
$$
C_0:=1, \quad C_k:=\frac{1}{k}\sum_{m=0}^{k-1}e_{k,m}C_m,
$$
and the formal series
$$
K_{p}(\underline{a}, \underline{b};z):= e^{z}
\sum_{k=0}^\infty C_k(\underline{a}, \underline{b}) z^{\nu-k}.
$$
By \cite[p.~283, Theorem]{joshi}, reproved in \cite[p.~113, 
Theorem~4.1, Eq.~(4.4)]{volkmer}, we have in fact
$$
C_k(\underline{a}, \underline{b})=\sum_{k_1\ge 0, \ldots, k_p\ge 0, K_p=k}\frac{(1-a_p)_{k_p}\prod_{j=1}^{p-1}(a_{j+1}+b_{j+1}-a_j)_{k_j}\prod_{j=1}^p(B_j+K_{j-1})_{k_j}}{\prod_{j=1}^p k_j!},
$$
where for every $j$, $B_j=\sum_{m=1}^j b_m$ and $K_j=\sum_{m=1}^j k_m$. It follows in particular that $K_{p}(\underline{a}, \underline{b};z)\in e^z z^\nu\mathbb Q[\underline{a}, \underline{b}][[1/z]]$.

In \cite[p.~212]{luke}, it is shown that as $z\to \infty$ in the sector $-\frac{3\pi}{2} <\arg(z) <\frac{\pi}{2}$, we have the asymptotic expansion 
$$
{}_{p}F_p \left[ 
\begin{matrix}
a_1, \ldots, a_{p}
\\
b_1, \ldots, b_p
\end{matrix}
; z
\right] \approx \frac{\prod_{j=1}^p \Gamma(b_j)}{\prod_{j=1}^p\Gamma(a_j)}\Big(L_{p}(\underline{a}, \underline{b};e^{i\pi}z) + K_{p}(\underline{a}, \underline{b};z)\Big), 
$$ 
while if 
$z\to \infty$ in the sector $-\frac{\pi}{2} <\arg(z) <\frac{3\pi}{2}$, we have 
$$
{}_{p}F_p \left[ 
\begin{matrix}
a_1, \ldots, a_{p}
\\
b_1, \ldots, b_p
\end{matrix}
; z
\right] \approx \frac{\prod_{j=1}^p \Gamma(b_j)}{\prod_{j=1}^p\Gamma(a_j)}\Big(L_{p}(\underline{a}, \underline{b};e^{-i\pi}z) + K_{p}(\underline{a}, \underline{b};z)\Big).
$$ 
These two expansions satisfy Definition~\ref{defiasy} above: they hold in a {\em large sector} bisected by any $\theta\in (-\pi,0)$, respectively any $\theta\in (0,\pi)$, and $L_p(\underline{a}, \underline{b};e^{\pm i\pi}z)$ and $e^{-z}K_p(\underline{a}, \underline{b};z)$ are $1$-summable in the direction $\theta$. Indeed, as already said, any hypergeometric series ${}_pF_p[\underline{a}, \underline{b};z]$ with rational parameters admits an asymptotic expansion in the sense of Definition~\ref{defiasy}, while Lemma~1 of~\cite{ateo} ensures that a function admits at most one expansion of this type in any given large sector bisected by a given direction.

 These asymptotic expansions are refined versions of Barnes and Wright's fundamental works \cite{barnes, wright} and are consequences of the general expansion of Meijer $G$-function \cite[Chapter~V]{luke}. Note that Meijer $G$-function is not related to Siegel's $G$-functions, though by specialization of its parameter the former provides examples of the latter. In the next two subsections, we provide more explicit expressions for the function $L_{p}(\underline{a}, \underline{b};z)$ under the assumption that the $a_j$'s and $b_j$'s are in $\mathbb Q\setminus \mathbb Z_{\le 0}$, in order to prove that all coefficients of the asymptotic expansion belong to $\h$.

\subsection{$R$ has simple poles}\label{ssec:simplepoles}

If the $a_j$'s are pairwise distinct modulo $\mathbb Z$, then the poles of $R(s)$ are simple, and we have 
$$
L_{p}(\underline{a}, \underline{b};z) = 
\sum_{j=1}^p\sum_{k=0}^\infty (-1)^k\frac{\Gamma(a_j+k)\prod_{i=1, i\neq j}^p\Gamma(a_i-a_j-k)}{k!\prod_{i=1}^p \Gamma(b_i-a_j-k)}z^{-a_j-k}.
$$
When the $a_j$'s and $b_j$'s are in $\mathbb Q\setminus \mathbb Z_{\le 0}$, $\frac{\prod_{j=1}^p \Gamma(b_j)}{\prod_{j=1}^p\Gamma(a_j)}L_{p}(\underline{a}, \underline{b};z)$ is thus equal to a finite sum 
$$
\sum_{j} z^{-a_j} f_{j}(z)
$$
with $f_{j}(z)\in \h[[1/z]]$. Note that the element $1/\pi\in \h$ appears through the use of the reflection formula $\frac{1}{\Gamma(s)}=\frac{1}{\pi}\sin(\pi s)\Gamma(1-s)$ for rational values of $s$.

\subsection{$R$ has multiple poles}\label{ssec:multiplepoles}

We assume that the $a_j$'s and $b_j$'s are in $\mathbb Q\setminus \mathbb Z_{\le 0}$. Up to reordering the $a_j$'s, we can group them in $\ell$ groups as follows: for $m=0, \ldots, \ell-1$, we have
$$
a_{j_m+1}, a_{j_m+2}, \ldots, a_{j_{m+1}} \; \textup{equal mod} \,\mathbb Z, \quad a_{j_m+1} \textup{ the smallest one in the group},
$$
where the $a_{j_m}$ are pairwise distinct mod $\mathbb Z$ for $m=1, \ldots, \ell $, and $0=j_0<j_1<j_2<\dots <j_\ell=p$.

Then, for every $j\in \{j_m+1, \ldots, j_{m+1}\}$, we have
$$
\Gamma(a_j+s)=(a_{j_m+1}+s)_{a_{j}-a_{j_m+1}}\Gamma(a_{j_m+1}+s).
$$
Set $d_{m}:=j_{m}-j_{m-1}\ge 1$, $c_m:=a_{j_{m-1}+1}$ and 
$$
P(s):=\prod_{m=0}^{\ell-1}\left(\prod_{j=j_m+1}^{j_{m+1}}{(a_{j_m+1}+s})_{a_j-a_{j_m+1}}\right) \in \mathbb Q[s].
$$
Hence, 
$$
R(s)=P(s)\Gamma(-s)\frac{\prod_{m=1}^\ell \Gamma(c_m+s)^{d_m}}{\prod_{m=1}^p\Gamma(b_m+s)}.
$$
To compute the residue of $R(s)z^{-s}$ at $s=-c_n-k$ for given $n\in \{1, \ldots, \ell\}$ and $k\in \mathbb Z_{\ge 0}$, we write 
$$
\Gamma(c_n+s)=\frac{\Gamma(c_n+s+k+1)}{(c_n+s)_k(c_n+s+k)}
$$
and define
$$
\Phi_{c_n,k}(s):=z^{-s}P(s)\Gamma(-s)\frac{\prod_{m=1, m\neq n}^\ell \Gamma(c_m+s)^{d_m}}{\prod_{m=1}^p\Gamma(b_m+s)}\cdot \frac{\Gamma(c_n+s+k+1)^{d_n}}{(c_n+s)_k^{d_n}}
$$
which is holomorphic at $s=-c_n-k$. We thus deduce from
$$
R(s)z^{-s} = \frac{\Phi_{c_n,k}(s)}{(c_n+s+k)^{d_n}}
$$
that 
$$
\textup{Residue}\big(R(s)z^{-s},s=-c_n-k\big)=\frac{1}{(d_n-1)!}\Phi_{c_n,k}^{(d_n-1)}(-c_n-k).
$$
It follows that $\frac{\prod_{j=1}^p \Gamma(b_j)}{\prod_{j=1}^p\Gamma(a_j)}L_{p}(\underline{a}, \underline{b};z)$ is equal to a finite sum 
$$
\sum_{j,\ell} z^{-a_j}\log(1/z)^{\ell} f_{j,\ell}(z)
$$
with $f_{j,\ell}(z)\in \h[[1/z]]$. This concludes the proof of Theorem \ref{thhyp}.

\section{Asymptotic expansion of the  hypergeometric $E$-func\-tion ${}_pF_q(z^{q-p+1})$} \label{sec:2bis}

The goal of this section is to prove the following result, which generalizes Proposition \ref{thhyp}. We have chosen to present and prove this proposition in an independent part because the case $p=q$ is one of the steps in the proof of Proposition \ref{connecconst} below, towards the proof of the general case $q\ge p\ge 1$. Moreover Proposition \ref{thhyp} is slightly more precise than Theorem \ref{thhyp3} in the sense that the non-explicited ``finite set'' below is reduced to $\{0\}$; however, this precision is not important for us.

\begin{theo} \label{thhyp3} Let $\theta\in(-\pi,\pi)$, and $F(z):= {}_{p}F_q[ 
a_1, \ldots, a_{p}; 
b_1, \ldots, b_p
; z^{q-p+1}] $ be an hypergeometric $E$-function with $q\ge p\ge 1$ and parameters $a_j\in \mathbb Q$ and $b_j\in \mathbb Q\setminus\mathbb Z_{\le 0}$. 
Assume that $\theta$ does not belong to some finite set that depends only on $F$. 
Then $F(z)$ has an asymptotic expansion 
\begin{equation}\label{eq:asympFgeneral}
F(x) \devasy \sum_{\rho\in\Sigma} e^{\rho x} 
 \sum_{\al\in S } \sum_{i\in T } \sum_{n=0}^{\infty} c_{\rho, \al,i,n}x^{-n-\al}\log(1/x)^i
\end{equation}
 in a large sector bisected by $\theta$, with $\Sigma\subset \{0,1\}$, $S\subset \mathbb Q$ and $T\subset \mathbb N$ both finite, and coefficients $c_{\rho, \al,i,n}$ in $\h$.
 \end{theo}

\begin{Remark} There exist well-known results in the litterature from which follows the existence of an  asymptotic expansion of a given ${}_{p}F_q[ 
a_1, \ldots, a_{p}; b_1, \ldots, b_p ; z]$ in any sector of opening $<2\pi$. See the references \cite{barnes, luke, wright} already cited in \S\ref{sec:2}. However, when $z$ is changed to $z^{q-p+1}$, these expansions hold only in sectors of opening $<2\pi/(q-p+1)$, so that when  $q>p$, these openings are $<\pi$. Unicity in a large sector, which is crucial for us, is thus not guaranteed anymore. The existence and unicity of the expansion \eqref{eq:asympFgeneral} is a consequence of Theorem \ref{thjep} because a ${}_{p}F_q(z^{q-p+1})$ series is an $E$-function when its parameters are rational numbers. We don't know explicit expressions for the coefficients $c_{\rho, \al,i,n}$ such as those obtained in \S\ref{sec:2} when $p=q$, but the result is enough for our purpose. 
\end{Remark}

\subsection{Connections constants of the $G$-function $(1/z)\cdot{}_{p+1}F_p(1/z)$}\label{ssec:2bis1}

We use the same notations as in  \S\ref{ssec:3.3} and \S\ref{sec:2}. Using unicity in large sectors of the asymptotic expansion of a ${}_{p+1}F_{p+1}(z)$ series with rational parameters, we shall  study the connection constants of its Laplace transform, which is a $(1/z)\cdot {}_{p+1}F_p(1/z)$ series.

More precisely, consider the hypergeometric $E$-function 
$$
f(z):={}_{p+1}F_{p+1}\left[ 
\begin{matrix}
a_1, \ldots, a_{p+1}
\\
1, b_1, \ldots, b_p
\end{matrix}
; z
\right]
$$
where $p\ge 0$ and $a_j\in \mathbb Q$,  $b_j\in \mathbb Q\setminus \mathbb Z_{\le 0}$. 
The Laplace transform of $f(z)$ is the $G$-function 
\begin{equation}\label{eq:defigprecise}
g(z):=\int_0^\infty f(t) e^{-zt} dt = 
\sum_{n=0}^\infty \frac{(a_1)_n\cdots (a_{p+1})_n n!}{(1)_n(1)_n(b_1)_n\cdots (b_p)_n}\frac{1}{z^{n+1}}
=\frac{1}{z}\cdot {}_{p+1}F_p \left[ 
\begin{matrix}
a_1, \ldots, a_{p+1}
\\
b_1, \ldots, b_p
\end{matrix}
; \frac{1}{z}
\right].
\end{equation}
The finite singularities of the differential equation satisfied by $g(z)$ are 0 and 1. Locally around
 $\rho \in \{0,1\}$, this equation has a  basis of solutions $(f_{j,k}^\rho(z-\rho))_{j,k}$ of the form \eqref{eqdeffjk} to which we connect $g(z)$ as in \eqref{gconnectbaseack}: 
$$
g(z) = \sum_{j=1}^{J(\rho)}\sum_{k=0}^{K(\rho,j)}\varpi_{j,k}^\rho f_{j,k}^\rho(z-\rho).
$$

As observed in the proof of Theorem \ref{thdevG}, we may assume in the expansions given by  \eqref{eqdeffjk} and \eqref{gconnectbaseack}  that for any $\rho$ the rational numbers $t_1^\rho$, \ldots, $t_{J(\rho)}^\rho$ are pairwise distinct, and that $g_{j,0}^\rho (0)\neq 0$ for any $\rho$ and any $j$. Then we have the following result.

\begin{prop} \label{connecconst} For any $\rho$, any $j$ and any $k$ such that $f_{j,k}^\rho (z-\rho)$ is not holomorphic at $z=\rho$, we have $\varpi_{j,k}^\rho \in \h$.
\end{prop}

We believe that $\varpi_{j,k}^\rho \in \h$ also when $f_{j,k}^\rho (z-\rho)$ is  holomorphic at $z=\rho$, but the proof of this result would require a new idea. However, the case where $f_{j,k}^\rho (z-\rho)$ is  holomorphic at $z=\rho$ is useless for proving Theorem \ref{thhyp3} (see Remark \ref{rem1} after Theorem \ref{thjep}).

\begin{proof} We first recall that $\widehat{\Gamma}(x):=1/\Gamma(x)$. 
We shall in fact prove Proposition \ref{connecconst} for any $G$-function $g(z)$ which is the Laplace transform of an $E$-function $f(z)$ with an asymptotic expansion of the form \eqref{eqasy1} in a  large sector bisected by any $\theta$ (except finitely many mod $2\pi$), with coefficients $c_{\rho, \alpha,i,n}\in\h$. Then Proposition \ref{thhyp} shows that the function $g(z)$ defined in \eqref{eq:defigprecise} has this property.

Let $\rho$, $j_0$ and  $k_0$ be fixed. We shall prove that $\varpi_{j_0,k_0}^\rho \in \h$; to begin with, we consider the case where $t_{j_0}^\rho\not\in\N$. By induction we may assume that for any $j$ such that  $t_{j_0}^\rho -   t_{j }^\rho $ is a positive integer and for any $k$, we have  $\varpi_{j ,k }^\rho \in \h$. The initial step of this inductive proof corresponds to the case where there is no such $j$, so that this assumption holds trivially. Let us denote by $\kappa$ the coefficient of $e^{\rho x} x^{- t_{j_0}^\rho - 1} \log(1/x) ^{k_0}$ in the asymptotic expansion of $f(x)$ in a large sector bisected by a given  $\theta$; by assumption we have $\kappa\in\h$. Now this asymptotic expansion is unique, and (except for finitely many values of $\theta$ mod $2\pi$) the following expression of $\kappa$ follows from Theorem \ref{thjep}:
$$
\kappa  = \sum_{1 \leq j \leq J(\rho) \atop t_{j_0}^\rho -   t_{j }^\rho \in\N} \sum_{k=k_0}^{K(j,\rho)} \varpi_{j,k}^\rho  
 \sum_{\ell = 0} ^{k -k_0} \frac{(-1)^{\ell}}{\ell! k_0!} 
\, \, \Gh^{(\ell)}(1-\{t_{j_0}^\rho \})\, \,  [x^{t_{j_0}^\rho -   t_{j }^\rho}](\eta_{j, k-k_0-\ell}^\rho (x)) 
$$
where $  [x^n](\eta_{j, i}^\rho (x)) $ is the coefficient of $x^n$ in the power series $\eta_{j, i}^\rho (x)\in \Qbar[[x]]$. It follows from the reflection formula $\widehat{\Gamma}(1-x)=\frac{1}{\pi}\sin(\pi x)\Gamma(x)$ that  $
 \Gh^{(\ell)}(1-\{t_{j_0}^\rho \})\in\h$ for any $\ell\in\N$, since $t_{j_0}^\rho\in\Q$. We have assumed that $ \varpi_{j,k}^\rho\in\h  $ for any pair $(j,k)$ such that $   t_{j_0}^\rho -   t_{j }^\rho$ is a positive integer, so that in the expression of $\kappa$ the total contribution of these terms belongs to $\h$. We recall that for any $j\neq j_0$ we have assumed that $t_j\neq t_{j_0}$, so that all terms with $j\neq j_0$ are included here. Proceeding by decreasing induction on $k_0$ we may also assume that $ \varpi_{j_0,k}^\rho\in\h  $  for any $k$ such that $k_0+1\leq k \leq K(j_0,\rho)$, or that there is no such $k$. Then the term corresponding to each such $k$ in the expression of $\kappa$  belongs also to $\h$. Since $\kappa\in\h$ we deduce that the only remaining term, namely the one with $(j,k)=(j_0,k_0)$, belongs to $\h$: we have $\varpi_{j_0,k_0}^\rho \lambda \in \h$ with $\lambda =  \frac1{  k_0!} 
\, \, \Gh (1-\{t_{j_0}^\rho \})\, \,  \eta_{j_0,0}^\rho (0)$. Now with the notation of \S \ref{subsecJEP} we have $ \eta_{j_0,0}^\rho (0) = y_{  t_{j_0}^\rho ,0}(0) g_{j_0,0}^\rho(0) = \frac{ \Gamma  (1-\{t_{j_0}^\rho \})}{ \Gamma  ( - t_{j_0}^\rho  )} g_{j_0,0}^\rho(0) $ so that $\lambda\neq 0$ and 
$\frac1\lambda =  \frac{  k_0!  \Gamma  ( - t_{j_0}^\rho  )}{ g_{j_0,0}^\rho(0)}\in \h$. Therefore we have proved that $\varpi_{j_0,k_0}^\rho  =\frac1\lambda \cdot \varpi_{j_0,k_0}^\rho \lambda $ belongs to $\h$.

\medskip

To conclude the proof, let us prove that $\varpi_{j_0,k_0}^\rho \in \h$ in the case where $t_{j_0}^\rho \in\N$; assuming that $f_{j_0,k_0}^\rho (z-\rho)$ is not holomorphic at $z=\rho$, we have $k_0\geq 1$ in this case. We denote   by $\kappa$ the coefficient of $e^{\rho x} x^{- t_{j_0}^\rho - 1} \log(1/x) ^{k_0-1}$ in the asymptotic expansion of $f(x)$ in a   large sector; note that the exponent of $
 \log(1/x)$ has changed with respect to the first case. Using Theorem \ref{thjep} we obtain:
$$
\kappa  = \sum_{1 \leq j \leq J(\rho) \atop t_{j_0}^\rho -   t_{j }^\rho \in\N} \sum_{k=k_0-1}^{K(j,\rho)} \varpi_{j,k}^\rho  
 \sum_{\ell = 0} ^{k -k_0+1} \frac{(-1)^{\ell}}{\ell! (k_0-1)!} 
\, \, \Gh^{(\ell)}(1-\{t_{j_0}^\rho \})\, \,  [x^{t_{j_0}^\rho -   t_{j }^\rho}](\eta_{j, k-k_0+1-\ell}^\rho (x)) .
$$
 As above we may assume  that for any $j$ such that  $t_{j_0}^\rho -   t_{j }^\rho $ is a positive integer and for any $k$, we have  $\varpi_{j ,k }^\rho \in \h$. By decreasing induction on $k$ we may also assume, in the same way, that $ \varpi_{j_0,k}^\rho\in\h  $  for any $k$ such that $k_0+1\leq k \leq K(j_0,\rho)$, or that there is no such $k$. Then as above, all terms with $j\neq j_0$ or $k\leq k_0+1$ belong to $\h$. Since $\kappa\in\h$ we deduce that
 $$
  \sum_{k=k_0-1}^{k_0} \varpi_{j_0,k}^\rho  
 \sum_{\ell = 0} ^{k -k_0+1} \frac{(-1)^{\ell}}{\ell! (k_0-1)!} 
\, \, \Gh^{(\ell)}(1-\{t_{j_0}^\rho \})\, \,   \eta_{j_0, k-k_0+1-\ell}^\rho (0)  \in\h.
$$
Now with the notation of \S\ref{subsecJEP}, $t_{j_0}^\rho \in\N$ implies that $y_{t_{j_0}^\rho,0}(x)$ is identically zero (as noticed in \cite{ateo}), so that $\eta_{j_0,0} ^\rho (x)$ is  identically zero too. Therefore the terms with $\ell = k-k_0+1$ vanish in the above sum. The only term that remains is the one with $k = k_0$ and $\ell=0$, so that we have  
 $$
   \varpi_{j_0,k_0}^\rho  
 \frac{1}{ (k_0-1)!} 
\, \, \Gh (1-\{t_{j_0}^\rho \})\, \,  \eta_{j_0,  1 }^\rho (0) \in\h.
$$
Letting $\lambda =  \frac{1}{ (k_0-1)!} 
\, \, \Gh (1-\{t_{j_0}^\rho \})\, \,  \eta_{j_0,  1 }^\rho (0) $, it is sufficient to prove that $\lambda\neq 0$ and $\frac1\lambda  \in \h$. We have $(k_0-1)! 
 \Gamma (1-\{t_{j_0}^\rho \}) \in\h$, and $ \eta_{j_0,  1 }^\rho (0) = y_{  t_{j_0}^\rho ,1}(0) g_{j_0,0}^\rho(0)$ using the fact that $ y_{  t_{j_0}^\rho ,0}(0) 
=0$ since $t_{j_0}^\rho \in\N$. We have assumed that $ g_{j_0,0}^\rho(0)$ is a non-zero algebraic integer; at last we have (see \cite[Eq. (4.6)]{ateo} for more details):
$$
 y_{  t_{j_0}^\rho ,1}(0)  =  \frac{\dd }{\dd t }\Big(\frac{\Gamma(1-\{t\})}
{\Gamma(-t )}\Big)_{| t=t_{j_0}^\rho } 
 =  \frac{\dd }{\dd t }\Big(   (-t)_{ t_{j_0}^\rho +1}
  \Big)_{| t=t_{j_0}^\rho } 
 = (-1)^{  t_{j_0}^\rho +1} (t_{j_0}^\rho )! \in \Qbar\etoile.$$
 This concludes the proof that $\lambda\neq 0$ and $\frac1\lambda  \in \h$, and that of Proposition \ref{connecconst}. \end{proof}

\subsection{Proof of Theorem \ref{thhyp3}}\label{ssec:2bis2}

We shall now transfer the result obtained in \S\ref{ssec:2bis1} to ${}_pF_q(z^{q-p+1})$ series.

Consider the hypergeometric $E$-function 
$$
F(z):={}_{p}F_q \left[ 
\begin{matrix}
a_1, \ldots, a_{p}
\\
b_1, \ldots, b_q
\end{matrix}
; z^{r}
\right]
$$
where $q\ge p\ge 1$, $r:=q-p+1\ge 1$, and $a_j\in \mathbb Q$,  $b_j\in \mathbb Q\setminus \mathbb Z_{\le 0}$. Consider also the hypergeometric $G$-function 
$$
g(z):=\frac{1}{z}\cdot{}_{q+1}F_q \left[ 
\begin{matrix}
a_1, \ldots, a_{p}, {1}/{r}, {2}/{r}, \ldots, {r}/{r}
\\
b_1, \ldots, b_q
\end{matrix}
; \frac1z
\right].
$$
The Laplace transform of $F(z)$ is 
\begin{equation*}
G(z):=\int_0^\infty F(t) e^{-zt} dt = 
\sum_{n=0}^\infty \frac{(a_1)_n\cdots (a_p)_n(rn)!}{(1)_n(b_1)_n\cdots (b_q)_n}\frac{1}{z^{rn+1}}
=
z^{r-1} g\Big(\frac{z^r}{r^r}\Big).
\end{equation*}
 The finite singularities of the (essentially hypergeometric) differential equation satisfied by $g(z)$ are 0 and 1. As in the previous sections
  (see \S \ref{subsecJEP}),  we connect $g(z)$ on a local  basis $(f_{j,k}^\rho(z-\rho))_{j,k}$ around $z=\rho \in \{0,1\}$: 
$$
g(z)=\sum_{j=1}^{J(\rho)}\sum_{k=0}^{K(j,\rho)} \varpi_{j,k}^\rho f_{j,k}^\rho(z-\rho)
$$
where
$$
f_{j,k}^\rho(z-\rho)=(z-\rho)^{t_j^\rho}\sum_{\ell=0}^k g_{j,k-\ell}^\rho (z-\rho)\frac{\log(z-\rho)^{\ell}}{\ell!}
$$
with $t_j^\rho\in  \mathbb Q$ and $g_{j,k-\ell}^\rho (z-\rho)\in \Qbar[[z-\rho]]$ holomorphic at $z=\rho$. Here, $\theta$ is chosen such that $\log(z)$ is defined with $-\pi-\theta<\arg(z)<\pi-\theta$ and $\log(z-1)$ is defined with $-\pi-\theta< \arg(z-1)<\pi-\theta$. For later use, we also impose that $\theta$ is such that the following property holds: the half-lines $L_0:=-e^{-i\theta}\mathbb R_+$ and $L_k:=re^{2i\pi k/r}-e^{-i\theta}\mathbb R_+$ ($k=1,2,\ldots, r$) are such that $L_j\cap L_k =\emptyset$ for $j\neq k$. 

We may assume that for any $\rho$ the rational numbers $t_1^\rho$, \ldots, $t_{J(\rho)}^\rho$ are pairwise distinct, and that $g_{j,0}^\rho (0)\neq 0$ for any $\rho$ and any $j$. Then Proposition~\ref{connecconst} proved in \S \ref{ssec:2bis1} shows that $ \varpi_{j,k}^\rho \in \h$ for any $j$, $k$, $\rho$, except maybe when $ f_{j,k}^\rho(z-\rho)$ is holomorphic at $z=\rho$ (i.e., when $k=0$ and $t_j^\rho\in\N$).

\medskip

The set of finite singularities of the differential equation $\mathcal{E}$ satisfied by $G(z)$ is 
$$
\{0,e^{2i\pi/r}r, e^{4i\pi/r}r, \ldots,  e^{2r i\pi/r}r \}.
$$
The function $G(z)$ can be continued to the simply connected cut plane $\mathcal{L}:=\mathbb C\setminus \cup_{k=0}^r L_k$. Since the minimal differential equations satisfied by $G(z)$ and $g(z)$ have the same order, 
$(z^{r-1}f_{j,k}^0(z^r/r^r))_{j,k}$ is a basis of solutions of $\mathcal{E}$ at $z=0$ while $(z^{r-1}f_{j,k}^1(z^r/r^r-1))_{j,k}$ is a basis of solutions of $\mathcal{E}$ at any one of the points $z=re^{2 i \ell \pi/r}$, $\ell =1, \ldots, r$. These bases are not necessarily  of the form \eqref{eqdeffjk}  but this is not essential  for our purpose. 

\medskip

Let us first consider the connection of $G(z)$ with the basis of solutions at $z=0$. Below, $\log(z)$ is defined  $-\pi-\theta<\arg(z)<\pi-\theta$, and we have $\log(z^r/r^r)=r\log(z)-r\log(r)$ for all $z\in \mathbb C$ such that $c_1(r)<\arg(z)<c_2(r)$, for some well chosen constants $c_1(r)<c_2(r)$. Note that $r\log(r)\in \h$. We then have, for all $z\in \mathcal L$ such that $c_1(r)<\arg(z)<c_2(r)$,
\begin{align*}
G(z)&=\sum_{j=1}^{J(0)}\sum_{k=0}^{K(j,0)} \varpi_{j,k}^0  z^{r-1}f_{j,k}^0(z^r/r^r) 
\\
&= \sum_{j=1}^{J(0)}\sum_{k=0}^{K(j,0)} \varpi_{j,k}^0 r^{-rt_j^0} z^{rt_j^0+r-1} \sum_{\ell=0}^k g_{j,k-\ell}^0(z^r/r^r)\frac{\log(z^r/r^r)^\ell}{\ell!} 
\\
&=\sum_{j=1}^{J(0)}\sum_{k=0}^{K(j,0)} \varpi_{j,k}^0 r^{-rt_j^0} z^{rt_j^0+r-1} \sum_{\ell=0}^k g_{j,k-\ell}^0(z^r/r^r)
S_\ell(\log(z)), 
\end{align*}
where $S_\ell[x]\in \h[x]$ is of degree $\ell$. 
Note that $r^{-rt_j^0}\in \Qbar$ and $g_{j,k-\ell}^0(z^r/r^r)\in \Qbar[[z]]$.  
Recall that, by Proposition \ref{connecconst}, we have $\varpi_{j,k}^0 \in \h$ except maybe when $k=0$ and $t_j^0\in\N$. We observe that in this special case,   $\varpi_{j,0}^0 $ appears in the formula above as the coefficient of a function holomorphic at 0, because  $t_j^0\in\N$ implies $rt_j^0+r-1\in\N$ and $g_{j,0}^0(z^r/r^r)$ is holomorphic at~$0$.  Hence, for all $z\in \mathcal{L}$ such that $c_1(r)<\arg(z)<c_2(r)$, we have  
\begin{equation}\label{eq:G0}
G(z)=\sum_{j=1}^{J(0)}\sum_{k=0}^{K(j,0)} \Omega_{j,k}^0  z^{rt_j^0+r-1} \sum_{\ell=0}^k G_{j,k-\ell}^0(z)\frac{\log(z)^\ell}{\ell!}
\end{equation}
with ${\Omega}_{j,k}^0\in \C$,   $G_{j,k-\ell}^0(z)\in \h[[z]]$ are holomorphic at $z=0$, and 
$  \Omega_{j,k}^0 \in \h$ for any pair $(j,k)$ such that $k\geq 1$ or $rt_j^0+r-1\not\in\N$. Now, Eq. \eqref{eq:G0}  extends to $\mathcal{L}$ by analytic contination, ie the assumption $c_1(r)<\arg(z)<c_2(r)$ can be dropped. 
  For our application, it is enough to know that the functions $G_{j,k-\ell}^0(z)$ are in $\h[[z]]$ and not necessarily $G$-functions. Moreover, we need no information on the nature of the constants $\Omega_{j,0}^0$ for $j$ such that $rt_j^0+r-1 \in\N$
because they are factors of terms in \eqref{eq:G0} that are holomorphic at 0, and therefore do not contribute to the asymptotic expansion of $F(z)$. Note that it may happen that some $t_{j_0}\notin \mathbb N$ is such that $rt_{j_0}^0+r-1\in \mathbb N$, in which case  we know that ${\Omega}_{j_0,0}^0$ is indeed in $\h$ but this information will not be useful to complete the proof of Theorem \ref{thhyp3}.

\medskip

Let us now consider the connection of $G(z)$ with the basis of solutions at $z=e^{2 \ell i\pi/r}r$. For simplicity, we shall assume that $\ell =0$ but the general case can be obtained similarily. Below, $\log(z -r)$ is defined with $-\pi-\theta< \arg(z-r)<\pi-\theta$. We have $\log(z^r/r^r-1)-\log(z-r)= 
Q(z-r)\in \h[[z-r]]$  for $z\in \mathbb C$ such that $\vert z-r\vert<\kappa$ and  $d_1(r)<\arg(z)<d_2(r)$, for some well chosen constants $\kappa$ and $d_1(r)<d_2(r)$. (The series $Q(z-r)$ may change if the angular sector $d_1(r)<\arg(z)<d_2(r)$ is changed to another one.) Hence, for $z\in \mathcal{L}$ such that $\vert z-r\vert<\kappa$ and  $d_1(r)<\arg(z)<d_2(r)$, we have
\begin{align*}
G(z)&=\sum_{j=1}^{J(1)}\sum_{k=0}^{K(j,1)} \varpi_{j,k}^1  z^{r-1}f_{j,k}^1(z^r/r^r) 
\\
&= \sum_{j=1}^{J(1)}\sum_{k=0}^{K(j,1)} \varpi_{j,k}^1  (z^r/r^r-1)^{t_j^1}z^{r-1} \sum_{\ell=0}^k g_{j,k-\ell}^1(z^r/r^r-1)\frac{\log(z^r/r^r-1)^\ell}{\ell!}  
\\
&=\sum_{j=1}^{J(1)}\sum_{k=0}^{K(j,1)} \varpi_{j,k}^1  (z-r)^{t_j^1}P_j(z-r) \sum_{\ell=0}^k \frac{1}{\ell!} g_{j,k-\ell}^1(z^r/r^r-1)
\big(\log(z-r) +Q(z-r)\big)^\ell 
\end{align*}
where $P_j(z-r)\in \Qbar[[z-r]]$. 
As above,  Proposition \ref{connecconst} yields   $\varpi_{j,k}^1 \in \h$ for any $j$, $k$, except maybe when $k=0$ and $t_j^1\in\N$.  Hence,  for all $z\in \mathcal L$ such that $\vert z-r\vert<\kappa$ and  $d_1(r)<\arg(z)<d_2(r)$, we have 
\begin{equation}\label{eq:Gr}
G(z)=\sum_{j=1}^{J(1)}\sum_{k=0}^{K(j,1)} {\Omega}_{j,k}^r \cdot (z-r)^{t_j^1} \sum_{\ell=0}^k G_{j,k-\ell}^r(z-r)\frac{\log(z-r)^\ell}{\ell!}
\end{equation}
with ${\Omega}_{j,k}^r\in \C$,  $G_{j,k-\ell}^r(z-r)\in \h[[z-r]]$ are holomorphic at $z=r$, and   ${\Omega}_{j,k}^r\in \h$ for any $j$, $k$ such that $k\geq 1$ or $t_j ^1 \not\in \N$. Now, Eq. \eqref{eq:Gr} extends to $\mathcal{L}$ by analytic continuation, ie the conditions  that $\vert z-r\vert<\kappa$ and  $d_1(r)<\arg(z)<d_2(r)$ can be dropped. For our application, it is enough to know that $G_{j,k-\ell}^r(z)$ are in $\h[[z]]$ and not necessarily $G$-functions. Moreover, we need no information on the nature of the constants $\Omega_{j,0}^r$ for $j$ such that $t_j ^1 \in \N$ 
because they are factors of terms in \eqref{eq:G0} that  are holomorphic at $z=r$, and therefore do
not contribute to the asymptotic expansion of $F(z)$  (using Remark \ref{rem1} stated after  Theorem~\ref{thjep}). Eq.~\eqref{eq:Gr} can be generalized to the other singularities $e^{2 i \ell \pi/r}r$ with obvious changes.

\medskip

Now, we use the analytic contination to $\mathcal L$ of Eqs. \eqref{eq:G0},  \eqref{eq:Gr}  and  analogues at $z=e^{2 \ell i\pi/r}r$ for other values of $\ell$ in the formula 
of Theorem~\ref{thjep}, and we deduce Theorem \ref{thhyp3}.

\section{Application to Siegel's problem: proof of  Theorem~\ref{thimplication}} \label{sec:fin}

We now complete the proof of Theorem \ref{thimplication}. Assume that Siegel's question has an affirmative answer, and let $\xi\in\GG$. Theorem \ref{thdevG} 
provides
an $E$-function $E(z)$ and a finite set $S\subset (-\pi,\pi)$ such that for any $\theta\in(-\pi, \pi)\setminus S$, $\xi$ is a coefficient of the asymptotic expansion of $E(z)$ in a large sector bisected by $\theta$. Now an affirmative answer to Siegel's question yields $n$  hypergeometric series $f_1,\ldots, f_n$ of the type ${}_p F_q(z^{q-p+1})$ (for various values of $q-p$, $p$ and $q$ such that $q\ge p\ge 1$) with rational parameters, $n$ algebraic numbers $\lambda_1,\ldots,\lambda_n$, and a polynomial $P\in\Qbar[X_1,\ldots,X_n]$, such that $E(z)=P(f_1(\lambda_1 z),\ldots, f_n(\lambda_n z))$. Choose $\theta\in(-\pi, \pi) $ outside a suitable finite set. Then Theorem \ref{thhyp3} implies that for any~$i$, the asymptotic expansion of $f_i(\lambda_i z)$ in a large sector bisected by $\theta$ has coefficients in $\h$. The same holds for $E(z)=P(f_1(\lambda_1 z),\ldots, f_n(\lambda_n z))$ because $\h$ is a $\Qbar$-algebra. Since such an asymptotic expansion is unique (see \S\ref{subsecasyexp}), the coefficient $\xi$ belongs to $\h$. This concludes the proof.

\section{A Siegel type problem for $G$-functions}\label{sec:siegelprobG}

We recall that  $\sum_{n=0}^\infty a_n z^n$ is a $G$-function if $\sum_{n=0}^\infty a_n z^n/n!$ is an $E$-function (in the sense of this paper). $G$-functions form a ring stable under $\frac{d}{dz}$ and $\int_0^z$; they are not entire in general, they have a finite number of singularities and they can be analytically continued in a cut plane with cuts at these singularities. Moreover, given any algebraic function $\alpha(z)$ over $\Qbar(z)$ holomorphic at 0 and any $G$-function $f(z)$, the functions $\alpha(z)$ and $f(z\alpha(z))$ are $G$-functions.

For any integer $p\ge 0$, the  hypergeometric series
\begin{equation}
\label{eq:Gfnhyp}
{}_{p+1}F_p \left[ 
\begin{matrix}
a_1, \ldots, a_{p+1}
\\
b_1, \ldots, b_p
\end{matrix}
; z
\right] := \sum_{n=0}^\infty \frac{(a_1)_n\cdots (a_{p+1})_n}{(1)_n(b_1)_n\cdots (b_p)_n} z^n
\end{equation}
is a $G$-function when $a_j\in \mathbb Q$ and $b_j\in \mathbb Q\setminus \mathbb Z_{\le 0}$ for all $j$; Galochkin's classification can be adapted to describe all the hypergeometric $G$-functions of type ${}_{p+1}F_p$. The simplest examples are ${}_1F_0[a; \cdot;z]=(1-z)^a$ ($a\in \mathbb Q$) and ${}_2F_1[1,1; 2; z]=-\log(1-z)/z$. If $a_j\in \mathbb Z_{\le 0}$ for some $j$, then the series reduces to a polynomial. Any polynomial with coefficients in $\Qbar$ of  functions of the form $\mu(z)\, {}_{p+1}F_p[a_1, \ldots, a_{p+1}; b_1,\ldots, b_p ;\lambda(z)]$ is a $G$-function, where the parameters
$a_j, b_j \in \mathbb Q$, and $\mu(z), \lambda(z)$ are algebraic over $\Qbar(z)$ and holomorphic at  $z=0$, with $\lambda(0)=0$. 

In the spirit of Siegel's problem for $E$-functions, it is natural to ask  the following question.
\begin{questno*}\label{questionG}
Is it possible to write any $G$-function as a polynomial with coefficients in $\Qbar$ of functions of the form $\mu(z)\cdot{}_{p+1}F_p[a_1, \ldots, a_{p+1}; b_1,\ldots, b_p ;\lambda(z)]$, with $p \ge 0$, $a_j, b_j \in \mathbb Q$, $\mu(z), \lambda(z)$ algebraic over $\Qbar(z)$ and holomorphic at $z=0$, and $\lambda(0)=0$?
\end{questno*}
We prove in this section a result similar to that for $E$-functions (recall that the inclusion ${\bf G} \subset \h$ is very unlikely: see \S \ref{subsecH}).
\begin{theo} \label{thimplication2}
At least one of the following statements is true:
\begin{enumerate}
\item[$(i)$] ${\bf G} \subset \h$; 
\item[$(ii)$] Question \ref{questionG} has a negative answer under the further assumption that the algebraic functions $\lambda$ have a common singularity in $\Qbar^* \cup \{\infty\}$ at  which they all tend  to $\infty$.
\end{enumerate}
\end{theo}

Our method seems inoperant if this further assumption on the $\lambda$'s is dropped. 
This problem is related to a con\-jec\-ture of Dwork \cite[p. 784]{dworkajm} concerning the classification of certain operators in $\Qbar(z)[\frac{d}{dz}]$ of order~2, which was disproved by Krammer \cite{krammer}; later on, Bouw-M\"oller \cite{bouw} gave other counter-examples, of a different nature.  Dwork's conjecture said that a globally nilpotent operator of order 2 either has a basis of algebraic solutions over $\Qbar(z)$ or is an algebraic pullback  of the hypergeometric equation for the ${}_2F_1$ with rational parameters. We will not define here globally nilpotent operators (see the references), but they are conjectured to coincide with $G$-operators, i.e.~operators in $\Qbar(z)[\frac{d}{dz}] \setminus \{0\}$ which are minimal for some non-zero $G$-function. It is known that operators ``coming from geometry" are $G$-operators and globally nilpotent, by results of Andr\'e~\cite[p.~111]{andreGfns} and Katz~\cite[Theorem 10.0]{katz} respectively. The Krammer and Bouw-M\"oller operators  ``come from geometry'' and in \cite[\S9]{bouw}, the authors even produced explicit $G$-functions solutions of their operators which are neither algebraic functions nor algebraic pullbacks of a ${}_2F_1$ with rational parameters. However, this does not rule out the  possibility that these $G$-functions could be polynomials in more variables in ${}_{p+1}F_p$ hypergeometric functions with various values of $p\ge 1$.

Finally, if there exist $\alpha\in \Qbar$, $\vert \alpha\vert<1$, and $s\in \mathbb N$ such that $\li_s(\alpha)\notin \h$, then the proof given below shows that $\li_s(\frac{\alpha z}{z-\alpha})$ provides a counter-example, of differential order~$s+1$, to Question~\ref{questionG} with the  restriction in Theorem~\ref{thimplication2}.

\subsection{$G$-values as connection constants of $G$-functions}  \label{ssec:Gconstconnec}
Given a non-zero $G$-function $f(z)$, let $L$ denote a non-zero   operator in $\Qbar(z)[\frac{d}{dz}]$ such that $Lf(z)=0$ and of minimal order for $f$. By standard results of Andr\'e, Chudnovsky and Katz recalled in \cite[\S3]{YA1} or \cite[\S4.1]{gvalues} (and already used in \S\ref{subsecJEP}), $L$ is fuchsian with rational exponents and, at any $\alpha\in \Qbar \cup\{\infty\}$, $L$ admits a $\mathbb C$-basis of solutions of the form
$$
F (z-\alpha) :=  \sum_{e\in E}\sum_{k\in K} (z-\alpha)^e \log(z-\alpha)^k f_{e,k}(z-\alpha)
$$
where $E\subset \mathbb Q$, $K\subset \mathbb N$ are finite sets, and the $f_{e,k}(z)$ are $G$-functions; if $\alpha=\infty$, $z-\alpha$ has to be replaced by $1/z$.  We call such  a basis an ACK basis of $L$ at $\alpha$. The determination of $\log(z-\alpha)$ is fixed but somewhat irrelevant to our purpose; 
 monodromy around  the singularities and $\infty$ of any solution  of $L$ produces further coefficients in $\Qbar[\pi]\subset {\bf G}$ for the local expansions of this solution at these points.

Given an element $F(z-\beta)$ of an ACK basis of $L$ at some point $\beta\in \Qbar \cup\{\infty\}$ and an ACK basis $F_1(z-\alpha), \ldots, F_{\mu}(z-\alpha)$ of $L$ at $\alpha\in \Qbar \cup\{\infty\}$, we can connect locally around $\alpha$ an analytic continuation of $F(z-\beta)$  (in a suitable cut  plane) to this ACK basis:
\begin{equation}\label{eqconnec}
F(z-\beta)=\sum_{j=1}^\mu \omega_j F_j(z-\alpha),
\end{equation}
where, following a general terminology, the complex numbers $\omega_1,\ldots, \omega_\mu$ are {\em connection constants}. In \cite[Theorem 2]{gvalues}, we proved that $\omega_1,\ldots, \omega_\mu$ are in fact in ${\bf G}$. We prove here a converse result.

\begin{theo}\label{theo:new} Let $\xi\in {\bf G}\setminus\{0\}$ and $\alpha\in \Qbar^* \cup\{\infty\}$. There exists a non-zero $G$-function $F(z)$ solution of $L\in \Qbar(z)[\frac{d}{dz}]\setminus\{0\}$, of minimal order for $F$, and  an ACK basis $F_1(z-\alpha), \ldots, F_{\mu}(z-\alpha)$ of $L$  at $\alpha$ such that the analytic continuation of $F(z)$ in a suitable cut plane is given by 
$$
\sum_{j=1}^\mu \omega_j F_j(z-\alpha),
$$
where $\omega_1=F(\alpha)=\xi$.
\end{theo}
\begin{proof} Let $\xi\in {\bf G}\setminus\{0\}$. We first assume that $\alpha\neq \infty$.  
By \cite[Theorem 1]{gvalues}, there exists a non-zero $G$-function $G(z)$  of radius of convergence $>\vert \alpha\vert$ such that $G(\alpha)=\xi$. Let $L\in \Qbar(z)[\frac{d}{dz}]\setminus\{0\}$ be of minimal order for $G$.  Let $F_1(z-\alpha), \ldots, F_{\mu}(z-\alpha)$ be an ACK basis of $L$ at $\alpha$. Up to relabeling the basis, we can  assume without loss of generality that there exists $\lambda\le \mu$ such that
\begin{equation}\label{eq:glfj}
G(z)=\sum_{j=1}^\lambda \omega_j F_j(z-\alpha)
\end{equation}
in a cut plane locally around $z=\alpha$, 
where the $\omega_j\in {\bf G}$ are all non-zero. Up to performing $\Qbar$-linear combinations  of the $F_j$'s, we can assume without loss of generality that for all $j\ge2$, $F_j(z-\alpha)=o(F_1(z-\alpha))$ locally around $z=\alpha$. (Doing so, $(F_1(z-\alpha), \ldots, F_{\mu}(z-\alpha))$ remains an ACK basis of $L$ at $\alpha$.)
Hence, 
$G(z)\sim \omega_1 F_1(z-\alpha)$ as $z\to  \alpha$ in the cut plane. Because $G(\alpha)=\xi\neq 0$ and $\omega_1\neq 0$, this implies that $\kappa:=F_1(0)\in \Qbar^*$ by \cite[Lemma 5]{gvalues}. 
Upon replacing $F_1$ with $\kappa^{-1}F_1$, we may assume that $\kappa=1$; then we have $\xi=G(\alpha)=\omega_1$.

To work around $\infty$, we fix  $\alpha\in \Qbar^*$ and keep 
  the same notations; we  consider now $H(z):= G(\frac {\alpha z}{z-\alpha})$ and $H_j(\frac1z):=F_j(\frac {\alpha z}{z-\alpha}-\alpha)$. 
Then $H$ is a non-zero $G$-function solution of a differential operator $M\in \Qbar(z)[\frac{d}{dz}]\setminus\{0\}$ minimal for $H$ (trivially deduced from $L$). Now, $H_1(\frac1z), \ldots, H_\mu(\frac1z)$ is an ACK basis of $M$ at $\infty$ and, by \eqref{eq:glfj}, we have
$$
H(z)=\sum_{j=1}^\lambda \omega_j H_j\Big(\frac1z\Big)
$$
still with $\omega_1=\xi$ and $H(z)\to \xi$ as $z\to \infty$  in a suitable cut plane. 
\end{proof}

\subsection{Analytic continuation of the hypergeometric series ${}_{p+1}F_p(z)$} \label{ssec:analyticont}
In this section, we prove the following result.
\begin{theo} \label{thhyp2} Let  $f(z)= {}_{p+1}F_p[ 
a_1, \ldots, a_{p+1}; 
b_1, \ldots, b_p
; z] $ be a hypergeometric series with parameters $a_j\in \mathbb Q$ and $b_j\in \mathbb Q\setminus\mathbb Z_{\le 0}$. Then, the analytic continuation of $f(z)$ to the domain defined by $\vert \arg(-z)\vert<\pi$ and $\vert z\vert >1$  is given by 
\begin{equation*}
\sum_{j=1}^{p+1}\sum_{\ell} z^{-a_j}\log(1/z)^\ell f_{j,\ell}(1/z)
\end{equation*}
where the sum over the integer $\ell$ is finite and each $f_{j,\ell}(z)\in \h[[z]]$ converges for $\vert z\vert <1$.
 \end{theo}
\begin{proof}
Let 
$$
R(s)=R(\underline{a}, \underline{b};s):= \frac{\prod_{j=1}^{p+1} \Gamma(a_j+s)}{\prod_{j=1}^p \Gamma(b_j+s)} \Gamma(-s).
$$
The poles of $R(s)$ are located at $-a_j-k$, $k\in \mathbb Z_{\ge 0}$, $j=1, \ldots, p+1$, and at $\mathbb Z_{\ge 0}$. We define the series
$$
M_{p}(\underline{a}, \underline{b};z) :=  
\sum_{j=1}^{p+1}\sum_{k=0}^\infty 
\textup{Residue}\big(R(s)(-z)^{s},s=-a_j-k\big),
$$
which converges for any $z$ such that $\vert \arg(-z)\vert<\pi$ and $\vert z\vert >1$.

Then the analytic continuation of $f$ to the domain defined by $\vert \arg(-z)\vert<\pi$ and $\vert z\vert>1$ is given  by  
$$
f(z) = \frac{\prod_{j=1}^p \Gamma(b_j)}{\prod_{j=1}^{p+1}\Gamma(a_j)}M_p(\underline{a}, \underline{b};z).
$$  
See the discussion in \cite[\S5.3.1]{luke}, and Eqs.~(5) 
and~(17) there in particular. 
The same method as in \S\ref{ssec:simplepoles} and \S\ref{ssec:multiplepoles} shows that 
$$
M_p(\underline{a}, \underline{b};z)=
\sum_{j=1}^{p+1}\sum_{\ell} z^{-a_j}\log(1/z)^\ell f_{j,\ell}(1/z)
$$
where the sum over the integer $\ell\ge 0$ is finite and each $f_{j,\ell}(z)\in \h[[z]]$ converges for $\vert z\vert <1$. This completes the proof. \end{proof}

Note that when the $a_j$'s are pairwise distinct mod $\mathbb Z$, the poles of $R(s)$ at $-a_j-k$, $k\in \mathbb Z_{\ge 0}$ are all distinct and we  simply have
\begin{multline*}
f(z) =  
\sum_{j=1}^{p+1} (-z)^{-a_j} \frac{\prod_{k=1, k\neq j}^{p+1}\Gamma(a_k-a_j)}{\prod_{k=1, k\neq j}^{p+1}\Gamma(a_k)}
\\ \times \frac{\prod_{k=1}^{p}\Gamma(b_k)}{\prod_{k=1}^{p}\Gamma(b_k-a_j)} {}_{p+1}F_p\left[ 
\begin{matrix}
a_j, 1-b_1+a_j,\ldots, 1-b_p+a_{j}
\\
1-a_1+a_j, \ldots * \ldots, 1-a_{p+1} +a_j
\end{matrix}
; -\frac1{z}
\right]
\end{multline*}
where $*$ means that the term $1-a_j+a_j$ is omitted.

\subsection{Proof of Theorem~\ref{thimplication2}}
We start with some general considerations. Let $F(z)$ be a $G$-function   and  $L\in \Qbar(z)[\frac{d}{dz}]\setminus\{0\}$ be its minimal operator. Given a cut plane and $\alpha\in\Qbar^* \cup\{\infty\}$, the local behaviour around $\alpha$ of the analytic continuation of $F$ is described by an ACK basis  of $L$ at $\alpha$. In particular,   if $\vert z\vert $ is large enough, the analytic  continuation of $F$  is of the form
\begin{equation}\label{eq:gcontinfini}
\sum_{e\in E}\sum_{k\in K} \sum_{n\ge 0} c_{e,k,n}z^{-e-n} \log(1/z)^k   
\end{equation}
where $c_{e,\ell,n}\in {\bf G}$, $E\subset \mathbb Q$, $K\subset \mathbb N$ are finite sets (recall that the  connection constants $\omega_j$ in Eq. \eqref{eqconnec} belong to ${\bf G}$, by \cite[Theorem 2]{gvalues}). 
Monodromy around $\infty$ and the singularities of $F$ produces  further coefficients in $\Qbar[\pi]\subset {\bf G}$ for the local expansions of $F$ at these points. Hence any analytic continuation of $F$ is in fact of the form \eqref{eq:gcontinfini} at $\infty$.

Let now $f(z)={}_{p+1}F_p[ 
a_1, \ldots, a_{p+1}; 
b_1, \ldots, b_p
; z] $ be a hypergeometric series with parameters $a_j\in \mathbb Q$ and $b_j\in \mathbb Q\setminus\mathbb Z_{\le 0}$. Let $\mu(z),\lambda(z)\in \overline{\Qbar(z)}$ be holomorphic at $z=0$, with $\lambda(0)=0$ and $\lambda(z)\to \infty$ as $z\to \infty$.  A more precise result than \eqref{eq:gcontinfini} can be obtained for the $G$-function $g(z):=\mu(z)f(\lambda(z))$. We first recall that the analytic continuations of $\mu(z)$ and $\lambda(z)$ in suitable cut planes admit convergent Puiseux expansions at $\infty$ of the form
$$
\sum_{n\ge -m} a_n z^{-n/d} \in \Qbar[[1/z^{1/d}]]
$$
for some integers $m\ge 0$ and $d\ge 1$. Moreover, there exists $n\le -1$ such that $a_n\neq 0$ for $\lambda$ because $\lambda(z)\to \infty$ as $z\to \infty$.  Using Theorem~\ref{thhyp2}, we then deduce that $g(z)$ admits an analytic continuation at $\infty$ of the form \eqref{eq:gcontinfini} with $c_{e,k,n}\in \h$. Again, because $\Qbar[\pi]\subset \h$, any analytic continuation of $g$ is of the form \eqref{eq:gcontinfini} at $\infty$ with $c_{e,k,n}\in \h$.

For $j=1, \ldots, N$, consider $g_j(z):=\mu_j(z)f_j(\lambda_j(z))$ where $f_j(z)$ is a ${}_{p+1}F_p$  hypergeometric series with rational parameters and  $\mu_j(z),\lambda_j(z)\in \overline{\Qbar(z)}$ are holomorphic at $z=0$, with $\lambda_j(0)=0$ and $\lambda_j(z)\to\infty$ as $z \to\infty$. For any polynomial $P\in\Qbar[X_1,\ldots,X_N]$, it follows from the above discussion that any analytic continuation of the $G$-function  $P(g_1(z),\ldots, g_N(z))$ is also of the form  \eqref{eq:gcontinfini} at $\infty$ with $c_{e,k,n}\in \h$. 

Let us now assume that Question \ref{questionG} has a positive answer when all the $\lambda$'s tend to $\infty$ as $z\to \infty$. Given $\xi\in {\bf G}\setminus\{0\}$, consider the non-zero $G$-function $F(z)$ given by Theorem~\ref{theo:new} for $\alpha= \infty$: in a suitable cut plane,  its analytic continuation is of the form \eqref{eq:gcontinfini} with $c_{0,0,0}=\xi$. On the other hand, we have  
$$
F(z)=P\big(g_1(z),\ldots, g_N(z)\big)
$$
in a neighborhood of $z=0$, where the polynomial $P$ and the $g_j$'s are as above. The properties of this specific analytic continuation of $F(z)$ and of those of the right-hand side imply that $\xi\in \h$. Hence ${\bf G}\subset \h$ in this case. 

If the $\lambda$'s all tend to $\infty$ as $z\to \beta \in \Qbar^*$, then the above argument can be adapted using Puiseux expansions of the $\mu$'s and $\lambda$'s of the form $\sum_{n\ge -m} a_n (z-\beta)^{n/d} \in \Qbar[[(z-\beta)^{1/d}]]$, a $G$-function $F$ given by Theorem \ref{theo:new} with $\alpha=\beta$ and an ACK basis at $\beta$.

\medskip

\noindent S. Fischler, Laboratoire de Math\'ematiques d'Orsay, Univ. Paris-Sud, CNRS, Universit\'e
Paris-Saclay, 91405 Orsay, France.

\medskip

\noindent T. Rivoal, Institut Fourier, CNRS et Universit\'e Grenoble Alpes, CS 40700, 38058 Grenoble cedex 9, France.

\medskip

\noindent Keywords: Hypergeometric series, Asymptotic expansions, Siegel's $E$- and $G$-functions.

\medskip

\noindent MSC 2010: 33C15 (Primary); 41A60, 11J91 (Secondary).

\end{document}